\documentclass[a4paper,12pt, reqno]{amsart}
\usepackage{amsfonts}
\usepackage{amsmath}
\usepackage{mathrsfs}
\usepackage{amsfonts,amssymb,amsmath,amsthm}
\usepackage{enumerate}
\usepackage[utf8]{inputenc}
\usepackage{xcolor}
\newtheorem{theorem}{Theorem}[section]
\newtheorem{proposition}[theorem]{Proposition}
\newtheorem{corollary}[theorem]{Corollary}
\newtheorem{lemma}[theorem]{Lemma}
\theoremstyle{remark}
\theoremstyle{definition}

\newtheorem{definition}[theorem]{Definition}

\numberwithin{equation}{section}

\def\R{\mathbb R}

\def\C{\mathcal{C}}
\def\L{\mathcal{L}}

\newcommand{\be} {\begin{equation}}
	\newcommand{\ee} {\end{equation}}
\newcommand{\bea} {\begin{eqnarray}}
	\newcommand{\eea} {\end{eqnarray}}
\newcommand{\Bea} {\begin{eqnarray*}}
	\newcommand{\Eea} {\end{eqnarray*}}
\newcommand{\pa} {\partial}

\newcommand{\al} {\alpha}

\newcommand{\ga} {\gamma}
\newcommand{\Ga} {\Gamma}
\newcommand{\om} {\omega}

\newcommand{\la} {\lambda}
\newcommand{\si} {\sigma}

\newcommand{\La} {\Lambda}

\newcommand{\noi} {\noindent}

\newcommand{\vp} {\varphi}
\newcommand{\var} {\varepsilon}

\newcommand{\Om} {\Omega}

\def\sqr#1#2{{\vbox{\hrule height.#2pt
			\hbox{\vrule width.#2pt height#1pt \kern#1pt
				\vrule width.#2pt}
			\hrule height.#2pt}}}
\def\square{\sqr74}
\def\qed{{\unskip\nobreak\hfil\penalty50\hskip1em
		\hbox{}\nobreak\hfil\square \parfillskip=0pt
		\finalhyphendemerits=0 \par\goodbreak \vskip8mm}}

\def\XXint#1#2#3{{\setbox0=\hbox{$#1{#2#3}{\int\limits\limits}$}
		\vcenter{\hbox{$#2#3$}}\kern-.5\wd0}}

\title[Global Existence of Reaction Diffusion systems on Evolving Domain]{Global Existence of Solutions to Reaction Diffusion Systems with Mass Transport Type Boundary Conditions on an Evolving Domain}

\author{Vandana Sharma} 
\address{Department of Mathematics,
Indian Institute of Technology Jodhpur, 
Rajasthan, India}
 \email{vandanas@iitj.ac.in}
  
  \author{Jyotshana V. Prajapat}
\address{	Department of Mathematics, University of Mumbai, Vidyanagari, Santacruz east, Mumbai 400 098, India}
\email{jyotshana.prajapat@mathematics.mu.ac.in}

\thanks{The first author acknowledge IIT Jodhpur for research grant support as SEED grant and infrastructural support}

\date{}

\begin{document}
	\begin{abstract}
We consider reaction diffusion systems where components diffuse inside the domain and react on the surface through mass transport type boundary conditions on an evolving domain. Using a Lyapunov functional and duality arguments, we
establish the existence of component wise non-negative global solutions.
\end{abstract}
\footnotetext{{\it Keywords:\/}
reaction-diffusion equations, mass transport, conservation of mass, global wellposedness, linear estimates, evolving domain.

{\it AMS Classification:\/}
Primary: 35K57, 35B45}
	
	\maketitle
\section{Introduction}
\label{sec:intro}
The reaction–diffusion mechanism is one of the simplest and most elegant pattern formation models. Turing (1952) \cite{RefWorks:26} first proposed the mechanism in the context of biological morphogenesis, showing that reactions between two diffusible chemicals (morphogens) could give rise to spatially heterogeneous concentrations through an instability driven by diffusion. Recently there has been ample of studies on models that involve coupled bulk surface dynamics \cite{RefWorks:102}, \cite{RefWorks:103},\cite{RefWorks:5},\cite{RefWorks:3}, \cite{RefWorks:11}. Hahn et al \cite{RefWorks:14} model the surfactant concentration by use of coupled bulk-surface model and R\"atz and R\"oger \cite{RefWorks:15}, \cite{RefWorks:16} studied the symmetry breaking in a bulk surface reaction diffusion model for signalling networks. In the former work, a reaction–convection diffusion is proposed that couples the concentration of the surfactants in the bulk and on the free surfaces while in the latter work, a single diffusion partial differential equation is formulated inside the bulk of a cell, while on the cell surface a system of two membrane reaction diffusion equations is formulated. 

Sharma and Morgan, \cite{RefWorks:1} worked on coupled reaction diffusion system with $m$ components in the bulk coupled with $n$ components on the boundary and under certain conditions established the local and global wellposedness of the model. They further established the uniform boundedness of the solution. Recent advances in mathematical modelling and developmental biology identify the important role of evolution of  domains during the reaction process as central in the formation of patterns, both empirically (Kondo and Asai \cite{RefWorks:17}) and computationally (Comanici and Golubitsky \cite{RefWorks:20} ; Crampin et al. \cite{RefWorks:21}). Experimental observations on the skin pigmentation of certain species of fish have shown that patterns evolve in a dynamic manner during the growth of the developing animal. Kondo and Asai \cite{RefWorks:17} describe observations on the marine angelfish Pomacanthus semicirculatus, where juveniles display a regular array of vertical stripes which increase in number during growth, with new stripes appearing in the gaps between existing ones as the animal doubles in length. Further in \cite{RefWorks:21}, Crampin et al investigated the sequence of patterns generated by a reaction–diffusion system on growing domain. They derived a general evolution equation to incorporate domain growth in reaction–diffusion models and consider the case of slow and isotropic domain growth in one spatial dimension.  The results that Crampin et al \cite{RefWorks:21} present, suggest at least in one-dimension, that growth may in fact stabilize the frequency-doubling sequence and subsequently that it may be a mechanism for robust pattern formation. Also, in this respect, many numerical studies, such as in  Barrass et al. \cite{RefWorks:22}; and Madzvamuse and Maini  \cite{RefWorks:23} ; Madzvamuse \cite{RefWorks:24}, of RDS's on evolving domains are available. We also observed that Kulesa et al. \cite{RefWorks:25} have incorporated exponential domain growth into a model for the spatio-temporal sequence of initiation of tooth primordia in the Alligator Mississippiensis. In the model, domain growth plays a central role in establishing the order in which tooth precursors appear.

A specific feature of reaction-diffusion patterns on growing domains is the tendency
for stripe patterns to double in the number of stripes each time the domain doubles in length, called mode-doubling. Since their seminal introduction by Turing \cite{RefWorks:26}, reaction-diffusion systems (RDS's) have constituted a standard framework for the mathematical modelling of pattern formation in chemistry and biology.  Numerous studies on the stability of solutions of RDS's on fixed domains are available, for example, Hollis et al. \cite{RefWorks:5} ; Rothe \cite{RefWorks:27}; Sharma and Morgan \cite{RefWorks:1}, but very little literature regarding the global wellposedness of solutions of RDS's on evolving domains. In direction of stability, Madzvamuse et al.\cite{RefWorks:28} provides a linear stability analysis of RDS's on continuously evolving domains, and Labadie \cite{RefWorks:29} examines the stability of solutions of RDS's on monotonically growing surfaces. Chandrashekar et al \cite{RefWorks:30} showed that RDS fulfils a restricted version of certain stability conditions, introduced by Morgan \cite{RefWorks:3} for fixed domain, then the RDS fulfills the same stability conditions on any bounded spatially linear isotropic evolution of the domain. They prove that, under certain conditions, the existence and uniqueness for a RDS on a fixed domain implies the existence and uniqueness for the corresponding RDS on an evolving domain. This is, to our best knowledge, the first result that holds independently of the growth rate and is thus valid on growing or contracting domains as well as domains that exhibit periods of growth and periods of contraction.  Again these models arise in the area of tissue engineering and regenerative medicine, elctrospun membrane which  are useful in applications such as filtration systems and sensors for chemical detection.

In \cite{RefWorks:99}, \cite{RefWorks:88}, \cite{RefWorks:19} and the references therein, the authors derive the equation for the reaction diffusion equation on a growing manifold with or without boundary. They imposed special growth conditions such as isotropic (including exponential) or anisotropic and studied the behaviour of solutions. More precisely, they studied  pattern formation on  a manifold beginning with an initial static pattern and compared it with the final  pattern after the manifold stops growing. The main focus of their work has been  stability analysis and numerical simulations to study the development of patterns with growth, on curved surfaces.

In this paper, we  prove the global existence for solutions of reaction diffusion system on a  domain in $\R^n$ evolving with time.   In \cite{RefWorks:5}, the authors had proved global existence and uniform boundedness for a class of two component reaction diffusion system where one of the components is given to be apriori bounded as long as the solutions exists. This was extended in \cite{V1}  to more general system involving two components and with Neumann boundary conditions using Lyapunov type functional for deriving the apriori esitmates.  Keeping in mind the possible applications to systems such as Brusselator (see Section  \ref{eg}), here we use techniques of \cite{V1} to obtain the global existence for a two component reaction diffusion system on an evolving domain in the case when one of the component remains apriori bounded. Extension of the estimates of Fabes-Riviere \cite{RefWorks:55} to a more general operator and construction of a suitable Lyapunov functional are crucial ingredients in our proof to obtain H\"older and  $L_p$ estimates. These results as well as the local existence is proved here  for  $m$ component system of reaction diffusion equation.  We also define a Lyapunov functional different from the one used in  \cite{{RefWorks:1}} and \cite{RefWorks:8} which can be used to obtain $L_p$ estimates for the $m$ component system, as in these references. Once this is done, the global existence for the general case of system of  $m$ components on evolving domain will follow from arguments similar to \cite{{RefWorks:1}}. 

 As in many of the existing  works,   we  consider here dilational  anisotropic as well as isotropic growth, though the arguments extend to a more general growth. Consider   compact domains  $ \Om_t \subset \R^n$, $t \geq 0$ with boundary $\pa \Om_t = \Gamma_t$ evolving according to the given law (flow)  $y(x,t)$ so that we can represent $\Om_t = y_t(\Om_0)= y_t(\Om)$, $t \geq 0$ where $\Om_0 = \Om$ is the initial domain. We assume that $y$ is a diffeomorphism and as in \cite{RefWorks:19}, it is separable in $t$ and $x$ variable.   In practice, one expects that for an arbitrary domain, at a future time $t$, the boundary $\Ga_t$ may begin self intersecting, or the domain $\Om_t$ may split. Here we are interested in modeling situations where the domain does not break up and the boundary evolves in such a way that $\Gamma_t$ continues to remain smooth. So, without loss of generality we assume that the domain and hence the boundary remain asymptotically close to a fixed domain, which we denote by $\Om_\infty$ with boundary $\Ga_\infty$, and that for each $t \geq 0$, $\Gamma_t$  is $C^{2+ \mu}$. 
Letting $c_i$ denote the concentration of the $i$-th component, we consider the  system of equations
\bea \label{c0}
\frac{\pa c_i}{\pa t} & = & d_i \Delta c_i + f_i(c_1, \ldots, c_m) \quad y \in  \Om_t, t \in (0,T), 1\leq i \leq m, \nonumber\\
d_i \frac{\pa c_i}{\pa \eta_t} & = &  g_i(c_1, \ldots, c_m) \quad y \in  \Gamma_t, t \in (0,T), 1\leq i \leq m, \\
c_i & = & c^0_i \quad y \in  \Om_0= \Om, t =0, 1\leq i \leq m \nonumber\eea 
where  $\eta_t$ denotes the outward unit normal vector to the boundary $\Ga_t$ and $\Om_0= \Om$ is the initial domain. 

 Note that our results can be  generalised to a manifold $(\Om, {\it g })$ with boundary where Laplacian $\Delta$ is replaced by the Laplace-Beltrami operator $\Delta_g$ corresponding to the Riemannian metric ${\it g}$.  This  will be done in our future work \cite {V2}, where we prove the global existence of solutions for volume-surface reaction diffusion systems on manifolds.

The first step is to transform  the system of equations (\ref{c0})  on $\Om_t$  to an
equivalent system on the initial domain $\Om$, as in \cite{RefWorks:99}, \cite{RefWorks:88}, \cite{RefWorks:19}.  The information on how the domain is evolving is captured in the diffusion term of the transformed equation, and generally, the evolution of domain is  described by a flow which is separable in time and spatial variables. 
We note that in  \cite{RefWorks:19}, a suitable transformation was used so that the  diffusion term in the resulting equation does not have time dependent term. Whereas we do analysis of the transformed equation with diffusion term  depending on the time variable.

The plan of paper is as follows. We begin by  the derivation of the equations on the evolving domain and reduction to a pull back system on the intial domain $\Om$ in Section 2.  Section 3 contains primary assumptions on the vector fields $f$ and $g$, and statements of our local and global existence results. In addition to quasi-positivity assumptions that guarantee the componentwise nonnegativity of solutions, we also assume polynomial bounds, and that the reaction vector fields satisfy a condition that is similar to the condition given in \cite{ RefWorks:8} and \cite{V1}. In Section 4 we discuss the H\"older estimates which will be useful in establishing the global wellposedness of the model on growing domain. Local existence is established in Section 5 and in Section 6 we develop a boot strapping process based upon duality estimates, and provide a proof of our global existence result. Section 7 contains a few examples.

\section{Equation for evolving domains}
\label{sec:Equation for evolving domains}
Here we  show how to reduce the system (\ref{c0}) on $\Om_t$ to a system on the fixed domain $\Om$. For simplicity of notations and keeping in mind practical applications, we show the derivation for domains $\Om_t \subset \R^3$.  Let $y_t: \Om \to \R^3 $ be    a one parameter family of diffeomorphisms such that $y_0 = Id$, the identity map and  $\Om _t = y_t(\Om)$ denote the domain evolving with time $t \geq 0$ such that $\Om_0 = \Om$.
 We  obtain a parametrization for   $\Om_t$ by writing $y \in \Om_t$  as $y = y(x,t) = (y_1(x,t), y_2(x,t), y_3(x,t)) = y_t(x)$ for $t \geq 0$, so that 
 \[y(x, 0) = x \in \Om_0 = \Om.\]
 If  $c$ denotes the chemical concentration in the domain $\Om_t$, then the diffusion process  for $c$  is driven by the equation
\be\label{1}\frac{d}{dt} \int\limits\limits\limits_{\Om_t} \left( c (y) \, d \Om_t \right) = D \int\limits\limits\limits_{\pa \Om_t} \nabla c(y) \cdot \nu_t \, d\si_t \ee
where  $d \Om_t = dy$ is the volume element in $\Om_t$,  $\si_t$ is a parametrization for $\pa \Om_t$ and $d \si_t$ is the surface area element for $\pa \Om_t$. Since $y$ is a diffeomorphism, we have $dy ={  \sqrt{det (Dy(x,t))} dx}$ and hence
\bea\label{02} 
&&\frac{d}{dt} \int\limits\limits\limits_{\Om_t} \left( c (y, t) \, d y \right) =  \frac{d}{dt} \int\limits_{\Om}  c (y(x,t), t) \, {  \sqrt{det (Dy(x,t))}}  \,dx \nonumber \\
 & = & \int\limits\limits\limits_{\Om} \left\{\frac{d y}{dt}(x,t) \cdot \nabla c(y(x,t), t) + \frac{d c}{dt}(y(x,t), t) \right\} \, { \sqrt{ det (Dy(x,t))}} \, dx \nonumber\\
 &&+  \int\limits\limits\limits_{\Om}c(y(x,t), t) \frac{d }{dt}\left( { \sqrt{ det (Dy(x,t)) }}\right)    \, dx, \eea
while using the Stokes theorem and change of variables, we see that
\bea \label{03}
D \int\limits\limits\limits_{\pa \Om_t} \nabla c(y, t) \cdot \nu_t \, d\si_t(y)  &=& D \int\limits\limits\limits_{\Om_t} \Delta c(y) \, dy\nonumber \\
&=& D \int\limits\limits\limits_{\Om} \Delta_t c((y(x),t),t)   \, {  \sqrt{det (Dy(x,t))}} dx
. \eea
Note, $\frac{\pa}{\pa y_i} c(y, t) = \sum\limits_{j=1}^3\frac{\pa x_j}{\pa y_i} \frac{\pa c}{\pa x_j} (y(x, t), t)$ so that 
\[\frac{\pa^2}{\pa y_i^2} c(y, t) = 
\sum_{j, k=1}^3\frac{\pa x_j}{\pa y_i} \frac{\pa x_k}{\pa y_i} \frac{\pa^2 c}{\pa x_j \pa x_k} (y(x, t), t)
+\sum_{j=1}^3\frac{\pa^2 x_j}{\pa y_i^2} \frac{\pa c}{\pa x_j} (y(x, t), t).\]
Thus, \be \Delta_t = \sum_{i=1}^3\sum_{j, k=1}^3\frac{\pa x_j}{\pa y_i} \frac{\pa x_k}{\pa y_i} \frac{\pa^2 }{\pa x_j \pa x_k}  + \sum_{i=1}^3\sum_{j=1}^3\frac{\pa^2 x_j}{\pa y_i^2} \frac{\pa }{\pa x_j}\ee
Combining equations (\ref{1})-(\ref{02}) we get that the concentration in the domain $\Om_t$ satisfies the equation
\bea\label{04}
&& \frac{d c}{dt}(y(x,t), t) + \frac{d y}{dt}(x,t) \cdot \nabla c(y(x,t), t))\nonumber\\
 & = &    D \Delta_t c((y(x),t),t) - \frac{1}{{  \sqrt{det (Dy(x,t))}}} \frac{d }{dt} \left( {  \sqrt{det (Dy(x,t)) }} \right) c(y(x,t), t)  \eea
for $ x \in \Om$ and $t \geq 0$. Define the function $u: \Om \times [0,T)$ as 
\[u(x,t) := c(y(x,t),t),\]
then $\frac{d}{dt}u(x,t)= \frac{d}{dt} c(y(x,t),t) = \frac{d c}{dt}(y(x,t), t) + \frac{d y}{dt}(x,t) \cdot \nabla c(y(x,t), t)$. The equation (\ref{04}) can now be written in terms of $u$ as
\be\label{05}
\frac{d u}{dt}(x,t) =  D \Delta_t u(x,t) - \frac{1}{{  \sqrt{det (Dy(x,t))}}} \frac{d }{dt} \left( {  \sqrt{det (Dy(x,t))  }}\right) u(x,t), ~\mbox{for}~ (x, t) \in \Om \times [0,t).\ee

In particular, for the flow $y(x,t):=  A(t) x$, $t \geq 0$ where $A(t):\R^3 \to \R^3$ is a family of diffeomorphism such that
$A(0) = Id$, the identity map so that $\Om_t = A(t) \Om$. As a special case we let 
\[A(t) = \left(\begin{array}{ccc} 
\la_1(t) &0 &0 \\
0 & \la_2(t) &0 \\
0 &0 & \la_3(t) \end{array} \right)
\] so that there is dilational growth- which is isotropic if $\la_1(t)= \la_2(t) = \la_3(t)$ and anisotropic otherwise.
The equation (\ref{05}) in this case is
\be \frac{d u}{dt}(x,t) =  D \Delta_t u(x,t) -  \frac{ ({  \sqrt{\la_1(t)\la_2(t)\la_3(t)}})'}{{  \sqrt{ \la_1(t)\la_2(t)\la_3(t)}}}  u(x,t), ~\mbox{for}~ (x, t) \in \Om \times [0,t)\ee 
with \[  \Delta_t = \sum\limits_{i=1}^3 \frac{1}{\la_i(t)^2} \frac{\pa^2 }{\pa x_i^2}  \]

Let $\si:[\al_1,\al_2] \times [\beta_1, \beta_2] \to \R^3$ be a parametrization of $\pa \Om = \Ga $ with
\[\si(\al, \beta) = ( x(\al, \beta), y(\al, \beta), z(\al, \beta))\] so that we can express  parametrization for  $\pa \Om_t = \Ga_t$ as $\si^t:[\al_1,\al_2] \times [\beta_1, \beta_2] \to \R^3$ with
\[\si^t(\al, \beta)= \si(\al, \beta, t) = ( x(\al, \beta, t), y(\al, \beta, t), z(\al, \beta, t))\] where for $t= 0$,
\[\si^0(\al, \beta) =\si(\al, \beta, 0) = \si(\al, \beta)  \] 
For a point $p_t  = \si^t(\al_0, \beta_0)\in \Ga_t$, the tangent plane is $T_{p_t} \Ga_t = span \{ \si^t_\beta (\al_0, \beta_0),\si^t_\alpha (\al_0, \beta_0) \} $ and the normal at this point is 
\be\label{normal}\nu_t(\al_0, \beta_0) = \si^t_\al \times \si^t_\beta (\al_0, \beta_0). \ee
The compatibility condition  is given by
\be D \nabla c \cdot \nu_t = G(u, v, t) \quad \mbox{on}\quad \Ga_t,\ee
which can be transformed to fixed boundary $\Ga$ as 
\be 
D \nabla_t u (\si(\al, \beta, t)) \cdot \nu = G(u, v, t) \quad \mbox{on}\quad \Ga\ee
where $\nu_t$ is defined as in (\ref{normal}).
With our parametrization,  $\Ga_t = y_t( \Ga)$ can be simply expressed as 
\[\si_t(\al, \beta) = y_t( x(\al, \beta), y(\al, \beta), z(\al, \beta))\] 
so that the area element for $\Ga_t$ is $\sqrt{\det y_t} d\si $ and hence the  pull back of this surface measure  on $\Ga$ will be  $\sqrt{\det A(t)} d\si $. We may also consider $A(t) = (a_{ij}(t))$, in which case we obtain a more complicated expression for $\Delta_t$ and rest of the arguments will follow similarly.


\section{Notations and Main results}


Throughout this paper, $n\geq 2$ and $\Omega$ is a bounded domain in $\mathbb{R}^n$ with smooth boundary $\Gamma$ ($\partial \Omega$) belonging to the class $C^{2+\mu}$ with $\mu>0$ such that $\Omega$ lies locally on one side of its  boundary, $\eta$ is the unit outward normal (from $\Omega$) to $\partial\Omega$, and $\Delta$ is the Laplace operator. In addition, $m, k, n,i$ and $ j$ are positive integers.
\subsection{Basic Function Spaces}
Let $\mathcal{C}$  be a bounded domain on $\mathbb{R}^m$ with smooth boundary such that $\mathcal{C}$ lies locally on one side of $\partial\mathcal{C}$. We define all function spaces on $\mathcal{C}$ and $\mathcal{C}_T=\mathcal{C}\times(0,T)$. 
$L_q(\mathcal{C})$ is the Banach space consisting of all measurable functions on $\mathcal{C}$ that are $q^{th}(q\geq 1)$ power summable on $\mathcal{C}$. The norm is defined as\[ \Vert u\Vert_{q,\mathcal{C}}=\left(\int_{\mathcal{C}}| u(x)|^q dx\right)^{\frac{1}{q}}\]
Also, \[\Vert u\Vert_{\infty,\mathcal{C}}= ess \sup\lbrace |u(x)|:x\in\mathcal{C}\rbrace.\]
Measurability and summability are to be understood everywhere in the sense of Lebesgue.

If $p\geq 1$, then $W^2_p(\mathcal{C})$ is the Sobolev space of functions $u:\mathcal{C}\rightarrow \mathbb{R}$ with generalized derivatives, $\partial_x^s u$ (in the sense of distributions) $|s|\leq 2$ belonging to $L_p(\mathcal{C})$.  Here $s=(s_1,s_2,$...,$s_n),|s|=s_1+s_2+..+s_n$, $|s|\leq2$, and $\partial_x^{s}=\partial_1^{s_1}\partial_2^{s_2}$...$\partial_n^{s_n}$ where $\partial_i=\frac{\partial}{\partial x_i}$. The norm in this space is \[\Vert u\Vert_{p,\mathcal{C}}^{(2)}=\sum_{|s|=0}^{2}\Vert \partial_x^s u\Vert_{p,\mathcal{C}} \]

Similarly, $W^{2,1}_p(\mathcal{C}_T)$ is the Sobolev space of functions $u:\mathcal{C}_T\rightarrow \mathbb{R}$ with generalized derivatives, $\partial_x^s\partial_t^r u$ (in the sense of distributions) where $2r+|s|\leq 2$ and each derivative belonging to $L_p(\mathcal{C}_T)$. The norm in this space is \[\Vert u\Vert_{p,\mathcal{C}_T}^{(2)}=\sum_{2r+|s|=0}^{2}\Vert \partial_x^s\partial_t^r u\Vert_{p,\mathcal{C}_T} \] In addition to $W^{2,1}_p(\mathcal{C}_T)$, we will encounter other spaces with two different ratios of upper indices, 
$W_2^{1,0}(\mathcal{C}_T)$, $W_2^{1,1}(\mathcal{C}_T)$, $V_2(\mathcal{C}_T)$, $V_2^{1,0}(\mathcal{C}_T)$, and $V_2^{1,\frac{1}{2}}(\mathcal{C}_T)$ as defined in \cite{RefWorks:65}.

We also introduce $W^l_p(\mathcal{C})$, where $l>0$ is not an integer, because initial data will be taken from these spaces. The space $W^l_p(\mathcal{C})$ with nonintegral $l$, is a Banach space consisting of elements of $W^{[l]}_p$ ([$l$] is the largest integer less than $ l$) with the finite norm\[\Vert u\Vert_{p,\mathcal{C}}^{(l)} =\langle u\rangle_{p,\mathcal{C}}^{(l)}+\Vert u\Vert_{p,\mathcal{C}}^{([l])} \]
where  \[\Vert u\Vert_{p,\mathcal{C}}^{([l])}=\sum_{s=0}^{[l]}\Vert \partial_x^s u\Vert_{p,\mathcal{C}} \] and 
\[\langle u\rangle_{p,\mathcal{C}}^{(l)}=\sum_{s=[l]}\left(\int_\mathcal{C} dx\int_\mathcal{C}{|\partial_x^s u(x)-\partial_y^s u(y)|}^p.\frac{dy}{|x-y|^{n+p(l-[l])}}\right)^\frac{1}{p}\]
$W^{l,\frac{l}{2}}_p(\partial\mathcal{C}_T)$ spaces with non integral $l$ also play an important role in the study of boundary value problems with nonhomogeneous boundary conditions, especially in the proof of exact estimates of their solutions.  It is a Banach space when $p\geq 1$, which is defined by means of parametrization of the surface $\partial\mathcal{C}$. For a rigorous treatment of these spaces, we refer the reader to page 81 of Chapter 2 of \cite{RefWorks:65}.

The use of the spaces $W^{l,\frac{l}{2}}_p(\partial\mathcal{C}_T)$ is connected to the fact that the differential properties of the boundary values of functions from $W^{2,1}_p(\mathcal{C}_T)$ and of certain of its derivatives, $\partial_x^s\partial_t^r$, can be exactly described in terms of the spaces $W^{l,\frac{l}{2}}_p(\partial\mathcal{C}_T)$, where $l=2-2r-s-\frac{1}{p}$.

For $0<\alpha,\beta<1$,  $C^{\alpha,\beta}(\overline{\mathcal{C}_T})$ is the Banach space of H\"{o}lder continuous functions $u$ with the finite norm \[ |u|^{(\alpha)}_{\overline{\mathcal{C}_T}}= \sup_{(x,t)\in{\mathcal{C}_T}} |u(x,t)|+[u]^{(\alpha)}_{x,\mathcal{C}_T}+[ u]^{(\beta)}_{t,\mathcal{C}_T}\]
where \[ [u]^{(\alpha)}_{x, {\overline{\mathcal{C}_T}}}= \sup_{\substack{(x,t),(x',t)\in {\mathcal{C}_T}\\ x\ne x'}}\frac{|u(x,t)-u(x',t)|}{|x-x'|^{\alpha}}\] and \[  [u]^{(\beta)}_{t, {\overline{\mathcal{C}_T}}}= \sup_{\substack{(x,t),(x,t')\in {\mathcal{C}_T}\\ t\ne t'}}\frac{|u(x,t)-u(x,t')|}{|t-t'|^\beta}\]
We shall denote the space $C^{\frac{\alpha}{2},\frac{\alpha}{2}}(\overline{\mathcal{C}_T})$ by $C^{\frac{\alpha}{2}}(\overline{\mathcal{C}}_T)$. $ C(\mathcal{C}_T,\mathbb{R}^n)$ is the set of all continuous functions $u: \mathcal{C}_T \rightarrow \mathbb{R}^n$, and
$ C^{1,0}(\mathcal{C}_T,\mathbb{R}^n)$ is the set of all continuous functions $u: \mathcal{C}_T\rightarrow \mathbb{R}^n$ for which $ u_{x_i}$ is continuous for all $1\leq i\leq n$.
$ C^{2,1}(\mathcal{C}_T,\mathbb{R}^n)$ is the set of all continuous functions $u: \mathcal{C}_T \rightarrow \mathbb{R}^n$ having continuous derivatives $u_{x_i},u_{{x_i}{x_j}}\ \text{and}\ u_t$ in $\mathcal{C}_T$.
Note that similar definitions can be given on $\overline{\mathcal{C}_T}$.

{\bf Assumptions on the system:\/} Let $\Om\subset \R^n$  with $C^{2+ \mu}$ boundary $\Ga$ for some $\mu > 0$. We consider the system 
\be\label{grow}\left.\begin{array}{rll}
 \frac{d u}{dt}(x,t) &= & {\mathcal L} u(x,t) +f_1(u,v),\quad~\mbox{in}~  \Om \times [0,t) \\
 \frac{d v}{dt}(x,t) &= & \tilde{\mathcal L}  v(x,t)+f_2(u,v), \quad
~\mbox{in}~ \Om \times [0,t)\\
\nabla_t u \cdot \eta=g_1(u,v) ,
 \ & & \nabla_t v \cdot \eta=g_2(u,v)\quad
~\mbox{on}~  \in \partial\Omega \times [0,t) \\
 u=u_0, \ & &v=v_0  \quad~\mbox{in}~ (x, t) \in \Om \times \{0\}
\end{array}\right\} \ee
with
the operator \be\label{oper} 
{\mathcal L } =  D \sum\limits_{i=1}^n \frac{1}{\la_i(t)^2} \frac{\pa^2 }{\pa x_i^2} -  \frac{ \left({  \sqrt{\prod\limits_{i=1}^n\la_i(t)}}\right)'}{ {  \sqrt{\prod\limits_{i=1}^n\la_i(t)}}}  =   D\Delta_t - a(t) \ee 
and 
\be \tilde{\mathcal L } =  \tilde D\Delta_t - a(t) \ee 
where \be a(t) = \frac{ \left({  \sqrt{\prod\limits_{i=1}^n\la_i(t)}}\right)'}{ {  \sqrt{\prod\limits_{i=1}^n\la_i(t)}}}.\ee 
 We assume that there exists constants $\La_1$, $\La_2 >0$ and $k_1$, $k_2 > 0$ such that 
\be \label{k}\left.\begin{array}{lll}
&&\La_1 \leq \frac{1}{\la_i^2(t)} \leq \La_2, \, i = 1, \ldots, n\\
 && k_1 \leq a(t) \leq k_2. 
\end{array} \right\}\ee
Here \be \nabla_t  = A(t)^{-1} \nabla =\left(\frac{1}{\la_1(t)} \pa_{x_ 1}, \frac{1}{\la_2(t)} \pa_{x_2},\ldots,  \frac{1}{\la_n(t)} \pa_{x_n}\right)\ee and $\eta$ is the unit outward normal vector on $\Ga$.

For sake of completeness, we also mention here the extension to $m $ components for evolving domains. That is, let ${\bf u} = (u_1, \ldots, u_m)$ be solution of the system 
\be\label{m-grow}\left.\begin{array}{rll}
 \frac{\pa u_i}{\pa t}(x,t) &= & {\mathcal L} u_i(x,t) +f_i(u),\quad~\mbox{in}~  \Om \times [0,t) \\
 \nabla_t u_i \cdot \eta &= &g_i(u) \quad
~\mbox{on}~  \in \partial\Omega \times [0,t) \\
 u_i & = & u_0,  \quad~\mbox{in}~ (x, t) \in \Om \times \{0\}
\end{array}\right\}  1\leq i\leq m, \ee
with
\be {\mathcal L }=  D \sum\limits_{i=1}^n \frac{1}{\la_k(t)^2} \frac{\pa^2 }{\pa x_k^2} -  \frac{ \left({  \sqrt{\prod\limits_{i=1}^n\la_i(t)}}\right)'}{ {  \sqrt{\prod\limits_{i=1}^n\la_i(t)}}}  =   D\Delta_t - a(t) \ee
$a(t)$ and $\nabla_t$ as before. We remark that throughout, $\mathbb{R}^m_{+}$ is the nonnegative orthant in $\mathbb{R}^m$, $m \geq 2$. Here we list the  assumptions required to prove our results for a general $m$ component system, with the understanding  that for $m=2$, we denote $u_1 = u$, $u_2 = v$ and $d_1=D$, $d_2 = \tilde D$.
\begin{enumerate}
\item[($V_{\text{N}}$)]  ${ u_0}=({u_0}_i)\in C^2(\overline\Omega)$ and is componentwise nonnegative on $\overline\Omega$. Moreover,  $u_0$ satisfies the compatibility condition\\
\[ d_i\frac{\partial {u_0}_i}{\partial \eta}=g_i({u_0})\quad \text{on}\ \Gamma, \quad \mbox{for all } i = 1, \ldots, m.\]

\item[{\bf ($V_{\text{F}}$)}] $f_i,g_i:\mathbb{R}^m\rightarrow \mathbb{R}$, for $i=1,\ldots, m$ are {locally Lipschitz}.
\item[($V_{\text{QP}}$)] $f$ and $g$ are {\em quasi positive}. That is, for each $i=1,...,m$, if $u\in\mathbb{R}^m_{+}$ with $u_i=0$ then $f_i(u), g_i(u)\geq 0$.  
 
 \item[($V_{\text{L1 }}$)] There exists constants $b_j > 0$ and $L_1 > 0$ such that 
 \[\sum_{j=1}^m b_j f_j(z),  \sum_{j=1}^m b_j g_j(z) \leq L_1
 \left( \sum_{j=1}^m  z_j + 1 \right) \quad \mbox{for all } z \in \R^m_+. \]

\item[($V_{\text{Poly}}$)] For $i=1,\ldots,m$, $f_i$ and $g_i$ are {\em polynomially bounded}.  That is, there exists $K_{fg}> 0$ and a natural number $l$ such that \[  f_i(u,v),  g_i(u,v)   \leq K_{fg}\left(u+v+1\right)^l \ \text{for all}\  (u,v)\in \mathbb{R}^m_{+}.\]
\end{enumerate}

Under the assumption that $f= (f_1, \ldots, f_m)$ and $g = (g_1, \ldots, g_m)$ are locally Lipschitz,  we are able to prove the following local existence result.
\begin{theorem}{\bf (Local Existence) }\label{local}
Suppose $(V_N)$, $(V_{F})$, and $(V_{QP})$ holds. Then there exists $T_{\max}>0$ such that $\left (\ref{grow}\right)$ has a unique, maximal, component-wise nonegative solution $(u,v)$ with $T=T_{\max}$.  Moreover, if $T_{\max}<\infty$ then 
\[\displaystyle \limsup_{t \to T^-_{\max}}\Vert u(\cdot,t)\Vert_{\infty,\Omega}+ \limsup_{t \to T^-_{\max}}\Vert v(\cdot,t)\Vert_{\infty,\Omega}=\infty. \]
\end{theorem}

We remark that this local existence result is true for $m$ components with $m \geq 2$ though we have indicated the proof here only for $m = 2$. The following result gives global existence of solutions of (\ref{grow}) in case we know that one of the components is bounded by a suitable function.
\begin{theorem}\label{martinthm}
	Suppose $(V_N)$, $(V_{F})$, $(V_{QP})$, $(V_{L1})$ and $(V_{Poly})$ hold, and let $T_{\max}>0$ be given in Theorem \ref{local}. If there exists a nondecreasing function $h\in C(\mathbb{R}_+,\mathbb{R}_+)$ such that $\|v(\cdot,t)\|_{\infty,\Omega}\le h(t)$ for all $0\le t<T_{\max}$, and there exists $K>0$ so that whenever $a\ge K$ there exists $L_a\ge 0$ so that 
	\begin{align}\label{nearlykoua}
	af_1(u,v)+f_2(u,v), \ ag_1(u,v)+g_2(u,v) \le L_a(u+v+1),\quad\text{for all}\quad (u,v)\in\mathbb{R}_+^2,
	\end{align}
	then  (\ref{grow}) has a unique component-wise nonegative global solution.  
\end{theorem}
In order to prove global wellposedness, we need H\"older estimates of the solution of the associated linearized problem. The H\"{o}lder estimates in Theorem 3.6 of \cite{RefWorks:1} are extended to a more general operator ${\mathcal L}$ described in (\ref{oper}). 

Thus,  consider the equation
\be\label{eqn3.6} \left.\begin{array}{rll}  
	\frac{\pa \vp}{\pa t} & =  { {\mathcal L} \vp   } + \theta \quad &\mbox{ for} \quad (x,t) \in \Om \times (0,T) \\
	d \nabla_t \vp \cdot \eta &=\vp_1 \quad &\mbox{ for} \quad (x,t) \in \Ga  \times (0,T) \\
	\vp(x, 0)& = \vp_0(x) \quad &\mbox{ for} \quad x \in \Om\end{array}\right\} \ee
where $\eta$ denotes the outward unit normal vector on $\Ga$. Then,
\begin{theorem}\label{T3.6} Let  $ p > n +1 $,  $T> 0$,  $\theta \in L_p(\Om \times (0,T))$, $ \vp_1 \in L_p(\Gamma \times (0,T))$ and $\vp_0 \in W^2_p(\Om)$ such that 
	\be d \frac{\pa \vp_0}{\pa \eta} = \vp_1(x,0) \quad \mbox{on} \quad \Gamma.  \ee
	Then,  there exists a unique weak solution $\vp \in V^{1,\frac 12}_2 (\Om_T)$ of (\ref{eqn3.6})   and a constant $C_{p, T, ||det A(t)||_\infty }>0$ independent of $\theta$, $\vp_1$, $\vp_0$    such that for  $0 < \beta < 1 - \frac{n+1}{p}$ 
	\be
	|\vp|^\beta_{\Om_T} \leq C_{p,T, ||det A(t)||_{\infty}} \left(|| \theta||_{p, \Om_T} + || \vp_1||_{p, \Ga_T} + || \vp_0||_{p, \Om}^{(2)}\right). \ee
\end{theorem}

\begin{theorem}{\bf (Global Existence) }\label{global}
Suppose $(V_N)$, $(V_{F})$, $(V_{QP})$, $(V_{Poly})$  and the condition $(V_{\text{L}})$ described below hold.
\begin{align}
  (V_{\text{L}})~~~ & \mbox{ There exists a constant } K>0,  \mbox{ such that for any  } a=(a_1,...,a_{m-1}, a_m) \in \R^{m} \nonumber\\
  & \mbox{ with } a_1, ...,a_{m-1}\geq K, \mbox{ and } a_m=1,   \mbox{ there exists a constant }  L_a\ge 0  \mbox{ such  that } \nonumber \\
&\sum_{j=1}^ma_j f_j(z), \sum_{j=1}^m a_jg_j(z)\leq L_a\left(\sum_{j=1}^m z_j+1\right) \quad  \text{for all} \quad  z\in\mathbb{R}^m_{+}.\nonumber
\end{align}
Then,  $(\ref{m-grow})$ has a unique component-wise nonegative global solution.
\end{theorem}

Note that defining $c(y,t) = u(A(t)^{-1}y, t)$ for $y \in \Om_t$, above results can be translated to  the solutions  ${\bf c }= (c_1, \ldots, c_m)$ of (\ref{c0}) as follows:
\begin{theorem}{ \bf (Local existence for evolving domain) }\label{local-c}
Suppose $(V_N)$, $(V_{F})$, and $(V_{QP})$ holds. Then there exists $T_{\max}>0$ such that $ (\ref{c0})$ has a unique, maximal, component-wise nonegative solution ${\bf c}$ with $T=T_{\max}$.  Moreover, if $T_{\max}<\infty$ then for all $i = 1, \ldots, m$,
\[\displaystyle \limsup_{t \to T^-_{\max}}\Vert c_i(\cdot,t)\Vert_{\infty,\Omega_t}=\infty. \]
\end{theorem}
\begin{theorem}\label{martinthm-c}
For $m = 2$, suppose $(V_N)$, $(V_{F})$, $(V_{QP})$, $(V_{L1})$ and $(V_{Poly})$ hold, and let $T_{\max}>0$ be given in Theorem \ref{local}. If there exists a nondecreasing function $h\in C(\mathbb{R}_+,\mathbb{R}_+)$ such that $\|c_2(\cdot,t)\|_{\infty,\Omega}\le h(t)$ for all $0\le t<T_{\max}$, and there exists $K>0$ so that whenever $a\ge K$ there exists $L_a\ge 0$ so that 
	\begin{align}\label{nearlykoua-c}
	af_1(u,v)+f_2(u,v), ag_1(u,v)+g_2(u,v) \le L_a(u+v+1),\quad\text{for all}\quad (u,v)\in\mathbb{R}_+^2,
	\end{align}
	then  (\ref{c0}) has a unique component-wise nonegative global solution ${\bf c} = (c_1, c_2)$.  
\end{theorem}

\begin{theorem}\label{global-c}{ \bf (Global existence for evolving domain) }
Suppose $(V_N)$, $(V_{F})$, $(V_{QP})$, $(V_{Poly})$  and the condition $(V_{\text{L}})$ described below hold.
\begin{align}
  (V_{\text{L}})~~~ & \mbox{ There exists a constant } K>0,  \mbox{ such that for any  } a=(a_1,...,a_{m-1}, a_m) \in \R^{m} \nonumber\\
  & \mbox{ with } a_1, ...,a_{m-1}\geq K, \mbox{ and } a_m=1,   \mbox{ there exists a constant }  L_a\ge 0  \mbox{ such  that } \nonumber \\
&\sum_{j=1}^ma_j f_j(z), \sum_{j=1}^m a_jg_j(z)\leq L_a\left(\sum_{j=1}^m z_j+1\right) \quad  \text{for all} \quad  z\in\mathbb{R}^m_{+}.\nonumber
\end{align}
Then,  $(\ref{c0})$ has a unique component-wise nonegative global solution.
\end{theorem}
In the next section we obtain estimates for the linearized problem associated to (\ref{m-grow}).

\section{H\"{o}lder Estimates: proof of Theorem \ref{T3.6}}
We begin by listing some of the results from \cite{RefWorks:65} which will be used in this as well as upcoming sections. Using the notations therein, we let 
\[\mathcal{L}(x,t,\pa_x,\pa_t) = \pa_t- \sum_{i,j=1}^n a_{i,j}(x,t)\frac{\partial^2 }{\partial x_i \partial x_j}+\sum_{i=1}^n a_{i}(x,t)\frac{\partial }{\partial x_i }+a(x,t) \] denote  a uniformly parabolic operator. Consider the Dirichlet problem 
\begin{align}\label{Dirchlet}
\mathcal{L} u & = f(x,t) \nonumber\\
u|_{\partial \Omega_T}=\Phi(x,t)& ~~
u|_{t=0}=\phi(x) 
\end{align}
then the  Theorem 9.1 from \cite{RefWorks:65} states
\begin{lemma}\label{lady} Let $q>1$. Suppose that the coefficients $a_{ij}$ of the operator $\mathcal{L}$ are bounded continuous function in $\overline {\C}_{T}$, while $a_i$ and $a$ have finite norms $\Vert a_i\Vert_{r,\C_T}^{(loc)}$ and $\Vert a\Vert_{s,\C_T}^{(loc)}$ respectively, where \[   
	r = 
	\begin{cases}
	\max (q,n+2) &\quad\text{for} \ q\ne n+2\\
	n+2+\epsilon &\quad\text{for}\ q=n+2 \\
	\end{cases}
	\] and \[   
	s = 
	\begin{cases}
	\max (q,\frac{n+2}{2}) &\quad\text{for} \ q\ne \frac{n+2}{2}\\
	\frac{n+2}{2}+\epsilon &\quad\text{for}\ q=\frac{n+2}{2} \\
	\end{cases}
	\]
	 Let $\partial \Omega\in C^{2+\mu}$ and $\epsilon>0$ is very small. Suppose the quantities $\Vert a_i\Vert_{r,\C_{t,t+\tau}}^{(loc)}$ and $\Vert a\Vert_{s,\C_{t,t+\tau}}^{(loc)}$ tends to zero for $\tau\rightarrow 0$.
	Then for any $f\in L_{q}(\C_T)$, $\phi\in W_q^{2-\frac{2}{q}}(\Omega)$ and $\Phi \in W_q^{2-\frac{1}{q}, 1-\frac{1}{2q}}(\partial \C_T)$ with 
	$q\ne \frac{3}{2}$, satisfying the case $q>\frac{3}{2}$ the compatibility condition of zero order
	\[ \phi|_{\partial \Omega}=\Phi_{t=0}\] system    has a unique solution $u\in  W^{2,1}_q(\C_T)$. Moreover it satisfies the estimates 
	\[ \Vert u\Vert_{q,Q_T}^{(2)}\leq c\left( \Vert f\Vert_{q,\C_T}+\Vert\phi\Vert_{q,\Omega}^{(2-\frac{2}{q})}+\Vert\Phi\Vert_{q,\partial\C_T}^{(2-\frac{1}{q})}\right)\]
\end{lemma}
Now for the Neumann problem
\begin{align}\label{Neu}
\mathcal{L} u & = f(x,t) \nonumber\\
u|_{t=0}&=\phi(x) \\
\sum\limits_{i=1}^n b_i(x,t) \pa_{x_i}u + b(x,t) u\mid_{\pa \C_T} &= \Phi(x,t) \nonumber
\end{align}
 where we assume $|\sum\limits_{i=1}^n b_i(x,t) \eta_i(x)|\geq \delta >0$ everywhere on $\pa \C_T$, $\eta$ denoting the unit outward normal vector to $\pa \C_T$. Then Neumann counterpart of above Lemma can be written as follows.
 \begin{lemma}
 Let $p > 1$ and suppose that $ \theta \in  L_p(\Om \times (0, T))$, $\vp_0 \in  W^{2 - \frac 2p}_p (\Om),$  and $\gamma \in  
 W_p^{1 -\frac 1p, \frac 12 - \frac{ 1}{
2p}} (\Gamma \times  (0, T))$ with $p \neq 3$. In addition, when $p > 3$ assume
 \[ d  \frac{\pa \vp_0}{\pa \eta} =  \ga(x, 0) \quad \mbox{ on } \Gamma.\]
Then (\ref{Neu}) has a unique solution $\vp  \in W^{2,1}_p(\C_T)$ and there exists $C$ independent of $\theta$, $\vp_0$ and $\ga$ such that
\[ ||\vp||_{p,\C_T}^{(2)} \leq C \left( ||\theta ||_{p, \C_T} + || \vp ||^{(2-\frac 2p)}_{p, \Om} + ||\ga ||^{(1- \frac 1p, \frac 12 - \frac{ 1}{2p})}_{p,\pa \C_T}\right)\]
 \end{lemma}

We will also need the following Corollary from \cite{RefWorks:65}.
\begin{corollary}\label{cor1}
If the conditions of Lemma \ref{lady} are fulfilled for $q>\frac{n+2}{2}$ then the solution of problem $(\ref{Dirchlet})$ satisfies a H\"older condition in $x$ and $t$. Moreover,  when $q>n+2$ then the derivatives of the  associated Neumann boundary value problem  will also satisfy H\"older condition in $ x$ and $t$. 
\end{corollary}


Next,  we will prove the H{\"o}lder estimates for the solution of the  linearized  Neumann problem associated to (\ref{m-grow}) corresponding to the operator $\Delta_t$. The ideas for these estimates were developed in Section 5 of \cite{RefWorks:1}  and here we adapt those techniques for our operator $\Delta_t$.  For this section, we will make further reduction by  writing $\tilde u =\left({  \sqrt{\prod\limits_{i=1}^n\la_i(t)}}\right) u$ so that $\tilde u$ solves the equation
\be \label{grow2}\frac{d \tilde u}{dt}(x,t) = D \Delta_t \tilde u(x,t) ~\mbox{for}~ (x, t) \in \Om \times [0,t)\ee
where \be  \Delta_t =  \sum\limits_{i=1}^n \frac{1}{\la_i(t)^2} \frac{\pa^2 }{\pa x_i^2}.  \ee
 With this reduction, instead of working with equation (\ref{eqn3.6}), it suffices to obtain estimates of the equation
\be\label{eqn3.6-0} \left.\begin{array}{rll}  
	\frac{\pa \vp}{\pa t} & =  { \Delta_t \vp   } + \theta \quad &\mbox{ for} \quad (x,t) \in \Om \times (0,T) \\
	d \nabla_t \vp \cdot \eta &=\vp_1 \quad &\mbox{ for} \quad (x,t) \in \Ga  \times (0,T) \\
	\vp(x, 0)& = \vp_0(x) \quad &\mbox{ for} \quad x \in \Om.\end{array}\right\} \ee
Infact, the results of this section hold for a general $\Delta_t = 	\sum_{i,j=1}^n a_{i,j}(t)\frac{\partial^2 }{\partial x_i \partial x_j}$ with a positive definite $A(t) = (a_{ij}(t) )$ where the coefficient matrix is function of $t$ alone. Extension of these results to more general operator will appear in a forthcoming work.

The proof of Theorem \ref{T3.6}  will follow arguing as in the proof of Theorem 3.6 of  \cite{RefWorks:1} for (\ref{eqn3.6-0}). Here we point out the necessary changes when we replace the usual Laplacian with $\Delta_t$ in the results used.  Firstly,  the following Lemma from Pg 351, \cite{RefWorks:65} gives  $W_p^{2,1}(\C_T)$ on the solutions of (\ref{eqn3.6}).
\begin{lemma}
Let $p > 1$. Suppose $\theta \in L_p(\C_T)$, $\vp_0 \in W^{(2-\frac 2p)}(\Om)$  and $\theta \in  W^{(1- \frac 1p, \frac 12-\frac{1}{ 2p})}(\Ga\times(0,T))$ with $p \neq 3$. In case $p> 3$ we further assume that \[    d \frac{\pa \vp _0}{\pa \eta}  = \vp_1 \quad \mbox{on}\quad {\Ga \times \{0\} }.\]
	Then equation (\ref{eqn3.6}) has a unique solution $\vp \in W_p^{2,1}(\C_T)$ and there exists a constant $ C(\Om, p, T)$, independent of $\theta$, $\vp_0$ and $\vp_1$ such that 
	\[
	||\vp||_{p,\C}^{(2)} \leq C(\Om_T, p, T) \left(  || \theta||_{p, \C} + || \vp_0||_{p, \Om}^{(2 - \frac 2p)} + || \vp_1||_{p,\C_T}^{(1- \frac 1p, \frac 12-\frac{1}{ 2p})}    \right)
	\]
\end{lemma}

Following pg. 356 of  \cite{RefWorks:65}, the fundamental solution of the operator $\Delta_t$ is given by 
\be
Z_0(x-y, y, t, s) =\frac{1}{|4\pi(t-s)|^{\frac n2} \sqrt{det  {   A(s)}}} \exp \left( - \frac{ \left<{ \tilde A}(s)(x-y), (x-y) \right>}{4\pi(t-s)}   \right)
\ee
where {  $\tilde A(s) = A(s)^{-1}$}.

For $0< \var < t$,  we can define the operator $J_\var(f)$ corresponding to our equation as 
\be\label{J} J_\var(f)(Q, t)  = \int\limits_0^{t-\var} \int\limits_\Gamma \frac{<{ \tilde A(s)}( y- Q), \eta_Q>}{\sqrt{det {   A(s)}}(t-s)^{\frac n2 +1}}
 \exp \left( - \frac{ \left<\tilde A(s)(Q-y), (Q-y) \right> }{4\pi(t-s)}   \right) f(s, y) \, d\si ds. \ee Noticing that the change of variables gives \[ J_\var(f) =   \int\limits_0^{t-\var} \int\limits_{\Gamma_t} \frac{<{ \tilde A(s)}( y- Q), \eta_Q>}{(t-s)^{\frac n2 +1}}  \exp \left( - \frac{ \left<(Q-y), (Q-y) \right> }{4\pi(t-s)}   \right) F(s, y) \, d\si ds \]  where $F (s, y) = f(s, A(s)^{-1} y)$, the estimates and properties of the operator $J_\var$ can  be summarized as follows.
\begin{proposition}{\bf (Fabes-Riviere) }
	Assume $\Om$ is a $C^1$ domain, $Q \in \Gamma$ and $\eta_Q$ denote the unit outward normal to $\Gamma$  at $Q$. For $ 0< \var < t$, let the  functional $J_\var(f)$ be defined as in $(\ref{J})$. Then \\
	1. for $1 < p < \infty$ there exists $C(p, || A||_\infty )> 0$ such that 
	\[ J(f)(Q,t) = \sup\limits_{0< \var < t} |J_\var(f)(Q,t)|\] satisfies
	\[
	||J(f)||_{L_p(\Ga \times (0,T))}\leq C(p, || A||_\infty ) ||f||_{L_p(\Ga \times (0,T))} \quad \mbox{for all } f \in L_p(\Ga \times (0,T));
	\]
	2.$\lim\limits_{\var \to 0^+} J_\var(f) = J(f)$ exists in $L_p(\Ga \times (0,T)$ and pointwise for almost every $(Q,t) \in \Ga \times (0,T)$ provided $f \in L_p(\Ga \times (0,T)$, $1 < p < \infty$;\\
	3. $c I + J$ is invertible on $L_p(\Ga \times (0,T)$ for each $1< p < \infty$ where $I$ is the identity operator and  $c\neq 0$ in $\R$.
\end{proposition}
Note that now the constants will also depend on the matrix $A$. 
For $Q \in \Gamma$, $(x,t) \in \C_T$ and $t > s$, define
\begin{align}\label{w}
W(t-s,x, Q)& = \frac{1}{\sqrt{det {  A(s)}}(t-s)^{\frac n2 + 1}} \exp \left( - \frac{  \left<\tilde A(s)(Q-y), (Q-y) \right>}{4\pi(t-s)} \right),  \nonumber\\
\mbox{  and }  g(Q,t)& = -2 ( -c_n I + J)^{-1} \ga(Q, t)  
\end{align}
where $c_n = \frac{\om_n H(0)}{2} $, $\om_n =$ surface area of unit sphere in $\R^n$ and 
$H(0) = \int\limits\limits\limits_0^\infty \frac{ 1}{s^{n/2 +1}}\exp( {- \frac{1}{4s}}) \, ds$.
Referring to the Theorem 2.4 in \cite{RefWorks:55} we  have the following definition.
\begin{definition}
	A function $\vp$ is a {\bf  classical  solution\/} of the system $(\ref{eqn3.6})$ with $d =1$ and $\ga \in L_p(\Ga \times (0,T))$ for $p > 1$ if and only if 
	\be\label{class}
	\vp(x,t) = \int\limits_0^t \int\limits_\Ga W(t-s, x, Q)g(Q,s) \,d\si \,ds \quad \mbox{for  all } (x, t) \in \C_T. 
	\ee\end{definition}

We claim that the classical solution $\vp $ of  (\ref{eqn3.6-0}) defined as in (\ref{class}) is H\"older continuous.  For $(x, T)$, $( y , \tau ) \in  \C_T$, $0 < \tau < T$,  consider the difference
\[
\vp(x,T) - \vp (y, \tau)  = \int\limits\limits\limits_0^t \int\limits\limits\limits_\Ga \left[W(T-s, x, Q) - W(\tau -s, y, Q) \right] g(Q,s) \,d\si \,ds. 
\]
The following three Lemmas provide the required estimates.  
\begin{lemma}
	Let $p > n+1$. Suppose $ (x, T)$, $( y, \tau ) \in \C$ with $ 0 < \tau < T$ and
	\[ {\mathcal R}^c = \{( Q, s) \in \Ga \times (0, \tau) : |x- Q| + |T-s|^{\frac 12}  < 2|x-y| + |T-\tau |^{\frac 12}\}.\]
	Then for $0< a< 1 - \frac{n+1}{p}$, there exists $C(p,n, \overline{\Om},T, ||A||_\infty) >0$ independent of $g \in L_p(\Ga \times (0,T)$ such that 
	\be \int\limits\limits\limits_{{\mathcal R}^c}   | \left(W(T-s, x, Q) - W(\tau -s, y, Q) \right) g(Q,s) |  \,d\si \,ds  \leq C \left(|x-y|+ | T-\tau|^{\frac 12}  \right)^a || g||_{p, \Ga \times [0,\tau]}.\ee
\end{lemma}

\begin{lemma} Let $p>n+1$. 
	Suppose $ (x, T)$, $( y, \tau ) \in \C$ with $ 0 < \tau < T$ and
	\[ {\mathcal R}= \{( Q, s) \in \Ga \times (0, \tau) :      2(|x-y| + |T-\tau |^{\frac 12})  <         |x- Q| + |T-s|^{\frac 12}  \}.\]
	Then for $0< a< 1 - \frac{n+1}{p}$, there exists $C(p,n, \overline{\Om},T, ||A||_\infty) >0$ independent of $g \in L_p(\Ga \times (0,T)$ such that 
	\be \int\limits_{ {\mathcal R}}| \left(W(T-s, x, Q) - W(\tau -s, y, Q) \right) g(Q,s) |  \,d\si \,ds\leq C \left(|x-y|+ | T-\tau|^{\frac 12}  \right)^a || g||_{p, \Ga \times [0,\tau]}.\ee
	
\end{lemma}

\begin{lemma}
	Let $p > n+1$ and suppose $ (x, T)$, $( y, \tau ) \in \C$ with $ 0 < \tau < T$. Then for $0< a< 1 - \frac{n+1}{p}$, there exists $C(p,n,\overline{\Om},T, ||A||_\infty) >0$ independent of $g \in L_p(\Ga \times (0,T)$ such that 
	\be\int\limits_\tau^T \int\limits\limits\limits_{\Ga} |W(T-s, x, Q)  g(Q,s)|\, d\si ds \leq C  (|T-\tau|)^a || g||_{p, \Ga \times [\tau, T]}.\ee

\end{lemma}
We refer to the proofs of Lemmas 5.5, 5.6 and 5.7 respectively in  \cite{RefWorks:1} which can be repeated verbatim for the above three Lemmas. Similar to Proposition 5.8 in \cite{RefWorks:1}, we have the following H\"{o}lder  estimates for the solution of (\ref{eqn3.6}). 

\begin{proposition}\label{Holder}
	Let $\gamma \in L_p(\Ga \times (0,T))$ for $ p > n+1$. Then a solution of $(\ref{eqn3.6}) $ is H\"{o}lder continuous on $\bar \Om \times (0,T)$ with H\"{o}lder exponent $0 < a< 1- \frac{n+1}{p}$ and there is a constant $K_p > 0$, depending on $p$, $\bar \Om$, $T$ and $||A||_\infty$, independent of $\vp_1$ such that 
	\be 
	|\vp(x,T) - \vp(y, \tau)| \leq K_p \left( |T-\tau|^{\frac 12} + |x-y|\right)^a ||\vp_1||_{p, \Ga \times (0,T)} 
	\ee
	for all $ (x,T), (y,\tau) \in \C$.
\end{proposition}

The proof of Theorem \ref{T3.6} can now be completed.

\section{ Local existence of the solution}
Here we illustrate the proof of local existence of solutions of (\ref{m-grow}) for the case $m=2$, which can be easily extended to $m$ component case. In order to prove local existence of the solution we need the following result. 
\begin{theorem}\label{lglobal}
	If $f= (f_1, f_2)$, $ g = (g_1, g_2)$ are  Lipschitz function then the $ (\ref{grow}) $ has a unique global solution.
	\end{theorem}
	\begin{proof} Here we sketch first few steps of the proof to indicate that the linear term can be controlled.	Let $T>0$ and $u_0, v_0 \in W_p^2(\Omega)\times W_p^2(\Omega)$ such that they satisfy the compatibility condition
		\[ \frac{\partial u_0}{\partial \eta}=g_1(u_0,v_0)\, \text{and} \ \frac{\partial v_0}{\partial \eta}=g_2(u_0,v_0)\]
		Set \[ X=\lbrace (u,v)\in C(\overline \Omega\times[0,T])\times C(\Omega\times [0,T]): u(x,0)=0\  \text{and} \ v(x,0)=0\ \text{for all } \ x\in \overline\Omega\rbrace\]
		Note that $(X,\Vert \cdot\Vert_{\infty})$ is a Banach space. Let $(u,v)\in X$ and  consider the problem
	\begin{align}\label{grow1}
& \frac{d U}{dt}(x,t) =  {\mathcal L} U(x,t)+f_1(u+u_0,v+v_0), &~\mbox{in}~  \Om \times [0,t) \nonumber\\
& \frac{d V }{dt}(x,t) =  \tilde{\mathcal L} V(x,t)+f_2(u+u_0,v+v_0), 
&~\mbox{in}~  \Om \times [0,t)\\
&\nabla_t U \cdot \eta=g_1(u+u_0,v+v_0) , \  \nabla_t V \cdot \eta=g_2(u+u_0,v+v_0) 
&~\mbox{on}~ \partial\Omega \times [0,t) \nonumber\\
&\nonumber U=u_0, \  V=v_0  &~\mbox{in}~ \Om \times \{0\}
\end{align} 

		From Lemma \ref{lady}, $(\ref{grow1})$ possess a unique  solution $(U,V)\in W_q^{2,1}(\Omega_T)\times W_q^{2,1}(\Omega_T)$. Furthermore, from embedding $(U,V)\in C(\overline\Omega_T\times [0,T])\times  C(\overline\Omega_T\times [0,T])$. 
		Define $ S:X\rightarrow X$ as \[  S(u,v)=(U-u_0, V-v_0)\] where $(U,V)$ solves $(\ref{grow1})$ . We will see $S$ is continuous and compact. 
		Using linearity, $(U-\tilde U, V-\tilde V)$ solves 
\be\label{grow3}\left.\begin{array}{rll}
 \frac{d }{dt}(U- \tilde U)(x,t) &= & {\mathcal L} (U-\tilde U)(x,t) +f_1(u+u_0,v+v_0)-f_1(\tilde u+u_0,\tilde v+v_0), ~\mbox{in}~ \Om \times [0,t) \\
 \frac{d }{dt}(V - \tilde V)(x,t)  &=& \tilde {\mathcal L} ( V-\tilde V)(x,t)+f_2(u+u_0,v+v_0)-f_2(\tilde u+u_0,\tilde v+v_0), ~\mbox{in}~  \Om \times [0,t)\\
\nabla_t(U-\tilde U) \cdot \eta &=& g_1(u+u_0,v+v_0)- g_1(\tilde u+u_0,\tilde v+v_0) ~\mbox{on}~  \partial\Omega \times [0,t),\\
\nabla_t( V-\tilde V)\cdot \eta &=  &  g_2(u+u_0,v+v_0)- g_2(\tilde u+u_0,\tilde v+v_0) 
~\mbox{on}~ \partial\Omega \times [0,t) \\
 U-\tilde U=0, \ && V-\tilde V=0  ~\mbox{in}~ \Om \times \{0\}.
\end{array}\right\}\ee 

	From Corollary \ref{cor1},  if $q>n+2$ then solution of $(\ref{grow3})$ is H\"{o}lder continuous. Therefore there exists $C$ independent of $f_i$ and $g_i$, $i = 1, 2$,  such that 
\be\begin{array}{lll}		\Vert U-\tilde U\Vert_{\infty,\Omega_T}+\Vert V-\tilde V\Vert_{\infty,\Omega_T}
&\leq &  C \left\{ \Vert f_1(u+u_0,v+v_0)-f_1(\tilde u+u_0,\tilde v+v_0)\Vert_{q,\Omega_T}\right. \nonumber\\
&&+ \Vert  f_2(u+u_0,v+v_0)-f_2(\tilde u+u_0,\tilde v+v_0)\Vert_{q,\Omega_T}\nonumber\\
&&+ \Vert  g_1(u+u_0,v+v_0)-g_1(\tilde u+u_0,\tilde v+v_0)\Vert_{q,\partial\Omega_T}\nonumber\\
&& \left.+ \Vert  g_2(u+u_0,v+v_0)-g_2(\tilde u+u_0,\tilde v+v_0)\Vert_{q,\partial\Omega_T} \right\}
\end{array}\ee
		Using boundedness of $\Omega$, there exists $\tilde C >0$ such that 
\be\begin{array}{lll}	
 \Vert U-\tilde U\Vert_{\infty,\Omega_T}+\Vert V-\tilde V\Vert_{\infty,\Omega_T}&\leq& \tilde C \left\{\Vert f_1(u+u_0,v+v_0)-f_1(\tilde u+u_0,\tilde v+v_0)\Vert_{\infty,\Omega_T} \right. \\
  && +\Vert f_2(u+u_0,v+v_0)-f_2(\tilde u+u_0,\tilde v+v_0)\Vert_{\infty,\Omega_T} \nonumber\\
  && +\Vert g_1(u+u_0,v+v_0)-g_1(\tilde u+u_0,\tilde v+v_0)\Vert_{\infty,\partial\Omega_T} \nonumber\\
  && \left.+\Vert g_2(u+u_0,v+v_0)-g_2(\tilde u+u_0,\tilde v+v_0)\Vert_{\infty,\partial\Omega_T} \right\}\\
\end{array}\ee
		Since $f_i$, $g_i$ $i = 1,2$ are Lipschitz, $S$ is continuous with respect to the sup norm. Now it remain to show that this $S$ is compact. Moreover, $p>n+2$ from Corollary \ref{cor1} imples that solution is infact H\"{o}lder continuous therefore $S$ maps bounded sets in $X$ to precompact sets, hence $S$ is compact with respect to sup norm.   The uniqueness of the solution follows by deriving the Gronwall's inequality on $\Om_T$ by arguments similar to as in the proof of Theorem 6.1 of \cite{RefWorks:1}. Since $T>0$ was arbitrary, we further conclude the existence of unique global solution.
	\end{proof} 

\noi {\bf Proof of Theorem \ref{local}:\/ }The proof of the theorem involves truncating the given functions $f$, $g$ so that the truncated functions are Lipschitz. Precisely, for each $r>k$, we define cut off functions $\phi_r\in C_0^{\infty}(\mathbb{R}^2,[0,1])$ and $\psi_r\in C_0^{\infty}(\mathbb{R}^2,[0,1])$ such that $\phi_r(z,w)=1$ when $\vert z\vert\leq r$ and $\vert w\vert\leq r$, and $\phi_r(z,w)=0$ for all $\vert z\vert >2r$ or $\vert w\vert\geq 2r$. Define $f_r=f\phi_r$ and $g_r=g\psi_r$. We also have $u_0\in W^{2}_p(\Omega)$ and $v_0\in W^{2}_p(\partial\Omega)$ with $p>n$ and $u_0$, $v_0$ satisfy the compatibility condition for $p>3$. Hence, from the Sobolev imbedding theorem, $u_0$ and $v_0$ are bounded functions, i.e., there exists $k>0$ such that $\Vert u_0(\cdot)\Vert_{\infty,\Omega}\leq k$ and $\Vert v_0(\cdot)\Vert_{\infty,\partial\Omega}\leq k$. Applying  Theorem \ref{lglobal}, we obtain global solution $(u_r, v_r)$ for each $r$. Then letting $r\to \infty$ we obtain the solution $(u,v)$ with required properties. We refer to \cite{RefWorks:1} for  details.

\qed

\section{Existence of global solution}

In this section we will prove global existence of solutions of the system (\ref{grow})  under given conditions. We begin by proving apriori estimates, in particular,  $L_1$ estimate for the solutions of $ (\ref{grow}).$
\begin{lemma}{\bf ($L_1$-estimates)}
Let $(u,v)$ be  the unique maximal nonnegative solution to $(\ref{grow})$ and suppose that $T_{max}<\infty$. If  $V_N$,  $ V_F$ and $V_{L1}$ holds, then there exists $C_1(D, \tilde D, L_1, k_2)$ such that 
	\be \Vert u(\cdot, t)\Vert_{1,\Omega}, \Vert u(\cdot, t)\Vert_{1,\Ga},\Vert v(\cdot, t)\Vert_{1,\Ga}\leq C_1(t)  \ \forall \ 0\leq t<T_{max} \ee
\end{lemma}
\begin{proof}
 Adding the two equations in $(\ref{grow})$ and integrating the equation  over $\Omega$, we get
\begin{align}\label{interior}
\frac{d}{dt}\int\limits\limits_{\Omega} (u+v)&=\int\limits\limits_{\Omega} D\Delta_t u+\int\limits\limits_{\Omega} \tilde D \Delta_t v+\int\limits\limits_{\Omega} a(t)  (u+v)+\int\limits\limits_{\Omega}(f_1(u,v) +f_2(u,v))\nonumber\\
 & \leq \int\limits\limits_{\Omega} (f_1(u,v)+f_2(u,v))+ \int\limits\limits_\Gamma (g_1(u,v)+g_2(u,v))+\int\limits\limits_{\Omega}a(t)(u+v)\nonumber\\
 &\leq \int\limits\limits_{\Omega}(L_1+k_2)\left({u+v+1}\right)+\int\limits\limits_\Gamma L_1\left({u+v+1}\right).
\end{align}
where recall $a(t) = \frac{(\lambda_1(t)\lambda_2(t)\lambda_(t))'}{(\lambda_1(t)\lambda_2(t)\lambda_3(t))}$ and $a(t) \leq k_2$ for all $t$ by assumption.
Fix $0<T<T_{max}$, $d>0$ a constant (to be chosen later), $L_1>0$ and consider the system 
\begin{align}
\varphi_t &=-d\Delta_t \varphi-(L_1+k_2)\varphi &  (x,t)\in\Omega\times (0, T) \nonumber\\
d\nabla_t\varphi\cdot\eta&=(L_1+k_2)\varphi+1& (x,t)\in \Gamma\times (0, T) \nonumber\\
\varphi&=\varphi_{T}& x\in\Omega, \ t=T.
\end{align}
Here, $\varphi_{T}\in C^{2+\gamma}(\overline\Omega)$ for some $\gamma>0$ is strictly positive and satisfies the compatibility condition
\[ d \nabla_T \varphi_{T} \cdot \eta=(L_1+k_2)\varphi_T+1\ \text {on} \ \Gamma\times \lbrace T\rbrace.\]
From Theorem 5.3 in chapter 4 of \cite{RefWorks:65}, $\varphi\in C^{2+\gamma, 1+\frac{\gamma}{2}}(\overline\Omega\times [0,T])$, and therefore $\varphi\in C^{2+\gamma, 1+\frac{\gamma}{2}}(\Gamma\times [0,T])$ . Moreover, arguing as in the  previous section, we conclude $\varphi\ge 0$.
Now, consider
\begin{align}
0&=\int\limits\limits_{0}^{T}\int\limits\limits_{\Omega} u(-\varphi_t-d\Delta_t\varphi-(L_1+k_2)\varphi)\nonumber\\
&= \int\limits\limits_{0}^{T}\int\limits\limits_{\Omega} \varphi (u_{t}-D\Delta_t u-(L_1+k_2)\int\limits\limits_0^T\int_\Omega u\varphi-\int\limits\limits_{0}^{T}\int\limits\limits_\Gamma  u d\nabla_t\varphi \cdot \eta+(D-d)\int\limits\limits_0^T\int_\Omega u\Delta_t\varphi\nonumber\\
&+\int\limits_{0}^{T}\int\limits\limits_\Gamma \varphi D\nabla_t u \cdot \eta+\int\limits_{\Omega}  u(x,0)\varphi(x,0)-\int\limits\limits_{\Omega} u(x,T)\varphi(\cdot, T)\nonumber\\
&\leq\int_0^T \int\limits_\Omega\varphi (f_1(u,v)-L_1 u)-\frac{D}{d}\int\limits_{0}^{T}\int\limits\limits_\Gamma ( u(L_1+k_2)\varphi+u)\nonumber+(D-d)\int\limits_0^T\int_\Omega u\Delta_t\varphi\\&+\int\limits\limits_{0}^{T}\int\limits\limits_\Gamma \varphi g_1(u,v)+\int\limits\limits_\Omega  u(x,0)\varphi(x,0)-\int\limits\limits_{\Omega} u(x,T)\varphi(\cdot, T).
\end{align}
For $v$ we have the similar equation with $f_1$ replaced by $f_2$ and $g_1$ replaced by $g_2$, i.e.,
\bea
0& \leq & \int_0^T \int\limits_\Omega\varphi (f_2(u,v)-L_1 v)-\frac{\tilde D}{d}\int\limits_{0}^{T}\int\limits\limits_\Gamma ( v(L_1+k_2)\varphi+u) +(\tilde D-d)\int\limits_0^T\int_\Omega v\Delta_t\varphi\nonumber\\
&&+\int\limits\limits_{0}^{T}\int\limits\limits_\Gamma \varphi g_2(u,v)+\int\limits\limits_\Omega  v(x,0)\varphi(x,0)-\int\limits\limits_{\Omega} v(x,T)\varphi(\cdot, T).
\eea

Summing these equations, and making use of $(V_{L1})$ and choosing $d=\min\lbrace D,\tilde D\rbrace$, gives
\begin{align}\label{eqn2}
&\int\limits_0^T\int_\Gamma (u+v) \le  \int\limits_0^T\int_\Gamma (u+v)(1+ (L_1 + k_2)  \vp) \le \nonumber\\
&\int\limits_0^T\int_\Omega L_1\varphi+\int\limits_0^T\int_\Gamma L_1\varphi+(D-d)\int\limits_0^T\int_\Omega u\Delta_t\varphi 
+(\tilde D-d)\int\limits\limits_0^T\int_\Omega v\Delta_t\varphi  \nonumber\\
&+\int\limits\limits_{\Omega} u_0(x)\varphi(x,0)-\int\limits_\Gamma  u(x,T)\varphi_T(x) +\int\limits_\Omega v_0(x)\varphi(x,0)-\int\limits_\Gamma  v(x,T)\varphi_T(x).
\end{align}
Since $\varphi_T$ is strictly positive, we can choose a $\delta > 0$ such that $\delta \le\varphi(x)$ for all $x\in\Omega$. Then (\ref{eqn2}) implies
\begin{align}\label{eqn2.1}
& \delta\int\limits_\Gamma  (u(x,T)+v(x,T))+\int\limits_0^T\int\limits_\Gamma  (u+v)\le \nonumber\\
& \int\limits_0^T\int\limits_\Omega L_1\varphi+\int\limits_0^T\int_\Gamma L_1\varphi+(D-d)\int\limits_0^T\int\limits_\Omega u\Delta_t\varphi +
(\tilde D-d)\int\limits_0^T\int\limits_\Omega v\Delta_t\varphi+\int\limits_{\Omega} ( u_0+v_0)\varphi(x,0).
\end{align}
Thus, there exist constants $C_1,C_2>0$, depending on $L_1$, $d$, $\varphi_T$, $u_0,v_0$, $D,\tilde D$, and at most exponentially on $T$, such that 
 \begin{align}\label{eqn2.2}
\delta\int\limits_\Gamma (u(x,T)+v(x,T))+\int\limits_0^T\int\limits_\Gamma (u+v)\le C_1+C_2\int_0^T\int\limits_\Omega (u+v)
\end{align}
Now, return to (\ref{interior}), and integrate both sides with respect to  $t$ to obtain 
\begin{align}\label{eqn2.3}
\int\limits_\Omega( u+v)dx\le L_1\left(\int\limits_0^t\int\limits_\Omega (u+v)+\int\limits_0^t\int\limits_\Gamma (u+v)+t|\Gamma|+t|\Omega|\right)+\int\limits_\Omega(u_0+v_0)
\end{align}
The second term on the right hand side of (\ref{eqn2.3}) can be bounded above by $L_1$ times the right hand side of (\ref{eqn2.2}). Using this estimate, and Gronwall's inequality, we can obtain a bound for $\int\limits_0^T\int_\Omega (u+v)$ that depends on $T$. Placing this on the right hand side of (\ref{eqn2.2}) gives a bound for $\int\limits_\Gamma ( u(x,T)+v(x,T))$ that depends on $T$. Applying this to the second integral on the right hand side of (\ref{interior}), and using Gronwall's inequality, gives the result.
\end{proof}
{\bf Remark:\/} The above proof can be imitated for $m$ components to obtain $L_1$ estimates for solutions of (\ref{m-grow}).

For sake of completeness of our arguments, we state below the Lemma 3.3 proved in \cite{V1}.
\begin{lemma}\label{omega}
Given $\gamma\geq 1$ and $\epsilon>0$, there exists $C_{\epsilon,\gamma}>0$ such that 
\be\label{best} \Vert v\Vert_{2,\partial\Omega}^2 \leq \epsilon\Vert\nabla v\Vert_{2,\Omega}^2+C_{\epsilon,\gamma}\Vert v^{\frac{2}{\gamma}}\Vert_{1,\Omega}^{\gamma}\ee
and 
\be \Vert v\Vert_{2,\Omega}^2 \leq \epsilon\Vert\nabla v\Vert_{2,\Omega}^2+C_{\epsilon,\gamma}\Vert v^{\frac{2}{\gamma}}\Vert_{1,\Omega}^{\gamma}\ee
for all $v\in H^1(\Omega)$.
\end{lemma}
{\bf Proof of Theorem \ref{martinthm}:}
If $T_{\max}=\infty$, then there is nothing to prove. So, assume $T=T_{\max}<\infty$. We first claim  that under the given assumptions, 
\be|| u||_{p, \Om_T} \leq C( p, h(T), L_1, D, \tilde D, \Om, ||det A(t)||_\infty). \ee 
 We can assume without loss of generality that $b_1=b_2=1$ in $(V_{L1})$. 
  Let $1<p<\infty$, set $p'=\frac{p}{p-1}$ and choose  $\xi\in L_{p'}(\Omega_T)$ such that \be \xi \ge 0 \mbox{ and }\|\xi\|_{p',\Omega_T}=1.\ee Furthermore, let $L_2\geq \max \lbrace \frac{\tilde D L_1}{D},L_1\rbrace$ and suppose $\varphi$ solves 
\begin{align}\label{phieq}
\varphi_t+D\Delta_t\varphi&=-L_1 \varphi-\xi \quad\text{in }\Omega_T,\nonumber\\
D \nabla_t\varphi\cdot \eta&=L_2\varphi \quad\text{on }\Ga_T,\\
\varphi&=0 \quad \text{in }\Omega\times\left\{T\right\}.\nonumber
\end{align}
Though (\ref{phieq}) may appear to be a backwards heat equation, the substitution $\tau=T-t$ immediately reveals that it is actually a forward heat equation. 
Arguing as in the proof of the  Theorem \ref{local}, we conclude that  $\varphi\ge 0$. In addition, from Lemma \ref{lady}, there is a constant $C(p, D, \tilde D,\Omega, L_1) >0 $, and independent of $\xi$ such that 
\be\label{phibound} \|\varphi\|_{p',\Omega_T}^{(2,1)}\le C.\ee
Multiply (\ref{phieq}) with $(u+v)$ and integrating by parts we have 
\begin{align}\label{duality1}
&\int\limits\limits\limits_0^T\int\limits_\Omega(u+v)\xi dxdt\nonumber\\
&= \int\limits\limits\limits_0^T\int\limits_\Omega(u+v)(-\varphi_t-D\Delta_t\varphi-L_1\varphi)dxdt\nonumber\\
&\le\int\limits_\Omega(u_0+v_0)\varphi(x,0)dx-\int\limits\limits\limits_0^T\int\limits_\Omega(u+v)D\Delta_t\varphi dxdt-\int\limits\limits\limits_0^T\int\limits_\Omega L_1(u+v)\varphi dxdt.
\end{align}
Multiplying equation (\ref{grow}), integrating by parts and  using $(V_{L1})$ we get
\begin{align}\label{duality1.1} 
&\int\limits\limits\limits_0^T\int\limits_\Omega\varphi(u_t+v_t)dxdt  \nonumber\\
&= \int\limits\limits\limits_0^T\int\limits_\Omega\varphi(f+g)dxdt
+ \int\limits\limits\limits_0^T\int\limits_{\Omega}\varphi (D \Delta_t u+\tilde D \Delta_t v) dxdt -\int\limits_0^T\int\limits_{\Omega}a(t)(u+v) dxdt \nonumber\\
& \leq   \int\limits_0^T\int\limits_\Omega\varphi  L_1(u+v+1)dxdt
+ \int\limits_0^T\int\limits_{\Omega}\varphi (D \Delta_t u+\tilde D \Delta_t v) dxdt 
-\int\limits_0^T\int\limits_{\Omega}a(t)(u+v) dxdt \nonumber \\
& \leq  \int\limits_0^T\int\limits_\Omega\varphi  L_1(u+v+1)dxdt
+ \int\limits_0^T\int\limits_{\Omega}( u D\Delta_t\varphi +\tilde D v\Delta_t\varphi ) dxdt \nonumber\\ &-\int\limits_0^T\int\limits_{\Omega}a(t)(u+v) dxdt -\int\limits_0^T\int\limits_\Ga L_2\varphi u-\int\limits_0^T\int\limits_\Ga \frac{\tilde D}{D}L_2\varphi v
\end{align}
Combining (\ref{duality1}) and (\ref{duality1.1}), we have
\begin{align}\label{duality1.2}
\int\limits_0^T\int\limits_\Omega(u+v)\xi dxdt &\le \int\limits_\Omega(u_0+v_0)\varphi(x,0)dx+\int\limits_0^T\int\limits_\Omega(\tilde D-D)v\Delta \varphi dxdt\nonumber\\
&+\int\limits_0^T\int\limits_\Omega L_1\varphi dxdt+\int\limits_0^T\int\limits_\Ga L_1 \varphi dxd\sigma-\int\limits_0^T\int\limits_{\Omega} a(t)(u+v) dxdt \nonumber\\
& -\int\limits_0^T\int\limits_\Ga L_2\varphi u-\int\limits_0^T\int\limits_\Ga \frac{\tilde D}{D}L_2\varphi v \nonumber\\
&\leq \int\limits_\Omega(u_0+v_0)\varphi(x,0)dx+\int\limits_0^T\int\limits_\Omega(\tilde D-D)v\Delta \varphi dxdt\nonumber\\
&+\int\limits_0^T\int\limits_\Omega L_1\varphi dxdt+\int\limits_0^T\int\limits_\Ga L_1 \varphi dxd\sigma-\int\limits_0^T\int\limits_{\Omega}a(t)(u+v) dxdt 
\end{align}
By assumption,  $\|v(\cdot,t)\|_{\infty,\Omega}\le h(t)$ and  (\ref{phibound}) implies $\| \varphi\|_{p',\Omega_T}^{(2,1)}\le C_0$. Also, integrating (\ref{phieq}) reveals that
\begin{align}
\int\limits_{\Omega}\varphi(\cdot,0)=-\int\limits_{0}^T\int\limits_{\partial\Omega}L_2\varphi+\int\limits_{0}^T\int\limits_{\Omega}L_1\varphi+\int\limits_{0}^T\int\limits_{\Omega}\xi.
\end{align} 
Therefore, $\|\varphi(\cdot,0)\|_{1,\Omega}$ can be bounded independent of $\xi$ by using the norm bound on $\varphi$ and the fact that  $\|\xi\|_{p',\Omega_T}=1$. In addition, the  trace embedding theorem implies $\|\varphi\|_{1,\Ga_T}$ can be bounded in terms of $\|\varphi\|_{p',\Omega_T}^{(2,1)}$, which can be bounded independent of $\xi$, for the same reason as above.

 Therefore, by applying duality to (\ref{duality1.2}), we see that
  \be\label{ubound} \|u\|_{p,\Omega_T} \leq C( p, h(T), L_1, D, \tilde D, \Om). \ee Also, since $1<p<\infty$ is arbitrary, we have this estimate for every $1<p<\infty$. Moreover, the sup norm bound on $v$, the $L_p(\Omega_T)$ bounds on $u$ for all $1<p<\infty$, and $(V_{Poly})$ imply we have $L_q(\Omega_T)$ bounds on $f(u,v)$ and $g(u,v)$ for all $1<q<\infty$. \\

Now,we use the bounds above and assumption (\ref{nearlykoua}) to show $\|u\|_{p,\Ga_T}$ is bounded  for all $1<p<\infty$.  To this end, we employ a modification of an argument given in \cite{RefWorks:8} for the case $m=2$. Suppose $p\in \mathbb{N}$ such that $p\ge 2$, and choose a constant  $ \Theta>\max\left\{K,\frac{D+\tilde D}{2\sqrt{D\tilde D}}\right\}$. 
For  $a,b\ge 0$ we denote $w^{(a,b)}:=u^av^b$ and define the polynomial
\be\label{2poly} 
P(u,v,p, \Theta^{\beta^2}) = \sum_{\beta=0}^p \frac{p!}{\beta!(p-\beta)!}\Theta^{\beta^2}w^{(\beta,p-\beta)}.
\ee
In general, to fix notation we let 
\be 
P(u,v,p, \Theta^{c(\beta)} ) = \sum_{\beta=0}^p \frac{p!}{\beta!(p-\beta)!}\Theta^{c(\beta)}w^{(\beta,p-\beta)}.
\ee
where $c(p)$ is a prescribed function of $p$. Note that
\begin{align}\label{Lprime1}
\frac{\pa P}{\pa t}&=\sum\limits_{\beta=0}^p \frac{p!}{\beta!(p-\beta)!}\Theta^{\beta^2}\left(\beta w^{(\beta-1,p-\beta)}u_t+(p-\beta)w^{(\beta,p-\beta-1)}v_t\right)\nonumber\\
&=\left(pv^{p-1}v_t+p\Theta^{p^2}u^{p-1}u_t\right)dx+X_1+X_2,
\end{align}
where
\begin{align}\label{Lprime2.1}
X_1&= \sum\limits_{\beta=1}^{p-1} \frac{p!}{(\beta-1)!(p-\beta)!}\Theta^{\beta^2} w^{(\beta-1,p-\beta)}u_t \nonumber\\
&=p\Theta v^{p-1}u_t dx+ \sum\limits_{\beta=2}^{p-1} \frac{p!}{(\beta-1)!(p-\beta)!}\Theta^{\beta^2} w^{(\beta-1,p-\beta)}u_t \nonumber\\
&= p\Theta v^{p-1}u_t dx+\sum\limits_{\beta=1}^{p-2} \frac{p!}{\beta!(p-\beta-1)!}\Theta^{(\beta+1)^2} w^{(\beta,p-\beta-1)}(u_1)_t \nonumber\\
&= p\Theta v^{p-1}u_t dx+ \sum\limits_{\beta=1}^{p-2} \frac{p!}{\beta!(p-\beta-1)!}\Theta^{\beta^2} w^{(\beta,p-\beta-1)}\Theta^{2\beta+1}u_t 
\end{align}
and
\begin{align}\label{Lprime2.2}
X_2&= \sum\limits_{\beta=1}^{p-1} \frac{p!}{\beta!(p-\beta-1)!}\Theta^{\beta^2} w^{(\beta,p-\beta-1)}v_t \nonumber\\
&= p\Theta^{(p-1)^2}u^{p-1}v_t dx+ \sum\limits_{\beta=1}^{p-2} \frac{p!}{\beta!(p-\beta-1)!}\Theta^{\beta^2} w^{(\beta,p-\beta-1)}v_t .
\end{align}

Combining (\ref{Lprime1})-(\ref{Lprime2.2}) gives
\be\label{Pprime}
\frac{\pa P}{\pa t}=\sum\limits_{\beta=0}^{p-1}\frac{p!}{\beta!(p-1-\beta)!}\Theta^{\beta^2}w^{(\beta,p-1-\beta)}\left(\Theta^{2\beta+1}u_t+v_t\right)
\ee
Clearly, above steps hold even if we differentiate with respect to any variable $x_i$, $i = 1, 2, 3$, i.e.,
\be \frac{\pa P}{\pa x_i }=\sum\limits_{\beta=0}^{p-1}\frac{p!}{\beta!(p-1-\beta)!}\Theta^{\beta^2}w^{(\beta,p-1-\beta)}\left(\Theta^{2\beta+1}\frac{ \pa u}{\pa x_i}+ \frac{\pa v }{\pa x_i}\right)
\ee
Using the fact that $(u,v)$ satisfies the equation (\ref{grow}).
\be \frac{\pa P}{\pa t} = \sum\limits_{\beta=0}^{p-1}\frac{p!}{\beta!(p-1-\beta)!}\Theta^{\beta^2}w^{(\beta,p-1-\beta)}\left(\Theta^{2\beta+1} (\L u + f_1(u,v)) + \tilde \L v + f_2(u,v)\right).
\ee
Integrating over $\Om$, we have
\be\label{Lprime4}
\int\limits\limits\limits_\Om \frac{\pa P}{\pa t} \,dx = I + II + III
\ee
where
\begin{align}\label{Lprime5}
I=\int\limits\limits\limits_\Omega\sum_{\beta=0}^{p-1}\frac{p!}{\beta!(p-1-\beta)!}\Theta^{\beta^2}w^{(\beta,p-1-\beta)}\left(\Theta^{2\beta+1}f_1(u,v)+f_2(u,v)\right)dx
\end{align}
and
\begin{align}\label{Lprime6}
II=\int\limits\limits\limits_\Omega\sum_{\beta=0}^{p-1}\frac{p!}{\beta!(p-1-\beta)!}\Theta^{\beta^2}w^{(\beta,p-1-\beta)}\left(\Theta^{2\beta+1}D\Delta_t u+\tilde D\Delta_t v\right)dx.
\end{align}
\begin{align}\label{Lprime100}
III &=-\int\limits_\Omega\sum_{\beta=0}^{p-1}\frac{p!}{\beta!(p-1-\beta)!}\Theta^{\beta^2}w^{(\beta,p-1-\beta)} a(t)\left(\Theta^{2\beta+1} u + v \right) dx \nonumber\\
&\leq  - k_1   C(p,\Theta) \int\limits_\Omega     (u+v)^p\, dx.
\end{align}
Choosing $\Theta \geq K$  and applying (\ref{nearlykoua}),  we have 
\begin{align}\label{Lprime101}
I &\leq \int\limits_\Omega\sum_{\beta=0}^{p-1}\frac{p!}{\beta!(p-1-\beta)!}\Theta^{\beta^2}w^{(\beta,p-1-\beta)}L_{\Theta}\left( u+ v + 1\right)dx \nonumber\\
&\leq  L_{\Theta} \int\limits_\Omega [(u + v)^p +  (u + v)^{p-1}] \, dx \nonumber\\
&\leq C( p, h(T), L_1, D, \tilde D, \Om,||det A(t)||_\infty, L_\Theta) \quad \mbox{from (\ref{ubound}).}
\end{align}
 While 
\begin{align}\label{Lprime7}
 II&=\int\limits_\Omega\sum_{\beta=0}^{p-1}\frac{p!}{\beta!(p-1-\beta)!}\Theta^{\beta^2}w^{(\beta,p-1-\beta)}\left(\Theta^{2\beta+1}D\Delta_t u+\tilde D\Delta_t v\right)dx 
 \nonumber\\
&=\int\limits_{\Gamma}\sum_{\beta=0}^{p-1}\frac{p!}{\beta!(p-1-\beta)!}\Theta^{\beta^2}w^{(\beta,p-1-\beta)}\left(\Theta^{2\beta+1}g_1(u,v)+g_2(u,v)\right)dx \nonumber\\
&  -\int\limits_\Omega\sum_{\beta=0}^{p-1}\frac{p!}{\beta!(p-1-\beta)!} \Theta^{\beta^2}<  \Theta^{2\beta+1}D \nabla_t w^{(\beta,p-1-\beta)} , \nabla_t u >   \, dx \nonumber\\
&-\int\limits_\Omega\sum_{\beta=0}^{p-1}\frac{p!}{\beta!(p-1-\beta)!}\Theta^{\beta^2}<  \tilde D \nabla_t w^{(\beta,p-1-\beta)} , \nabla_t v > \, dx.
\end{align} 
We have  $\nabla_t w^{(\beta,p-1-\beta)}  = \beta u^{\beta -1}  v^{(p-1-\beta)} \nabla_t u + u^\beta (p-1-\beta) v^{(p-2-\beta)} \nabla_t v$, hence we can write
\begin{align}
&\int\limits_\Omega\sum_{\beta=0}^{p-1}\frac{p!}{\beta!(p-1-\beta)!} \Theta^{\beta^2} \left(<  \Theta^{2\beta+1}D \nabla_t w^{(\beta,p-1-\beta)} , \nabla_t u >   + <  \tilde D \nabla_t w^{(\beta,p-1-\beta)} , \nabla_t v > \, dx\right) \, dx \nonumber\\ 
&\geq \int\limits_\Omega\sum_{\beta=0}^{p-2}\frac{p!}{\beta!(p-2-\beta)!}  \Theta^{\beta^2} w^{(\beta,p-2-\beta)} \sum\limits_{k=1}^3 \frac{1}{\la_k(t)^2}  \left<B(\Theta, D, \tilde D) \left(\begin{array}{c}  \pa_{ x_k} u\\ \pa_{ x_k} v \end{array}\right),  ( \pa_{ x_k} u, \pa_{ x_k} v )  \right>  \nonumber
\end{align}
where 
$$ B(\Theta, D, \tilde D)=\begin{pmatrix} D\Theta^{4\beta+4}&\frac{(D+\tilde D)}{2}\Theta^{2\beta+1}\\\frac{(D+\tilde D)}{2}\Theta^{2\beta+1}& \tilde D\end{pmatrix}.$$
Again choosing $\Theta$ sufficiently large so that the matrix $B(\Theta, D, \tilde D)$ is positive definite and recalling that $ 0< \La_1 \leq \frac{1}{\la_i^2(t)} \leq \La_2$, there exists $\alpha_{\Theta,p}>0$ such that
\begin{align}\label{Lprime07}
&\int\limits_\Omega\sum_{\beta=0}^{p-1}\frac{p!}{\beta!(p-1-\beta)!} \Theta^{\beta^2} \left(<  \Theta^{2\beta+1}D \nabla_t w^{(\beta,p-1-\beta)} , \nabla_t u >   + <  \tilde D \nabla_t w^{(\beta,p-1-\beta)} , \nabla_t v > \, dx\right) \, dx \nonumber\\ 
& \geq  \La_1^2\alpha_{\Theta,p}\int\limits_\Omega\sum_{\beta=0}^{p-2}\frac{p!}{\beta!(p-2-\beta)!}  \Theta^{\beta^2} w^{(\beta,p-2-\beta)} \sum\limits_{k=1}^3 ( \pa_{ x_k} u^2 + \pa_{ x_k} v^2 ) \nonumber \\
&\geq  \La_1^2\alpha_{\Theta,p}\int\limits_\Omega \left(|\nabla (u)^{p/2}|^2+|\nabla (v)^{p/2}|^2\right)dx
\end{align}
Substituting (\ref{Lprime100}), (\ref{Lprime101}),   (\ref{Lprime7}) and (\ref{Lprime07}) in (\ref{Lprime4}) we get 
\begin{align}\label{Lprime8.5}
&\frac{\pa P}{\pa t} +  \La_1^2\alpha_{\Theta,p}\int\limits_\Omega \left(|\nabla (u)^{p/2}|^2+|\nabla (v)^{p/2}|^2\right)dx +  k_1   C(p,\Theta) \int\limits_\Omega     (u+v)^p\, dx \nonumber\\
&\leq C( p, h(T), L_1, D, \tilde D, \Om, L_\Theta) + N_{p,\Theta,\Ga}\left[\int\limits_\Ga\left(u^p+v^p\right)d\sigma+1\right]. 
\end{align}
 for some constant $N_{p,\Theta,\Ga}$.
 Applying Lemma \ref{omega}  to the functions $u^{p/2}$ and $v^{p/2}$ with $\ga = p$ and using (\ref{best}),  there exists $\tilde N_{p,\Theta,\Omega}>0$ such that 
\begin{align}\label{Lprime8.7}
2N_{p,\Theta,\Gamma}\int_{\Gamma}\left(u^p+v^p\right)d\sigma \leq \alpha_{\Theta,p}\La_1\int\limits_\Omega \left(|\nabla (u)^{p/2}|^2+|\nabla (v)^{p/2}|^2\right)dx+\tilde N_{p,\Theta,\Gamma}\left(\int_{\Omega}(u+v) dx\right)^p
\end{align}
Adding $(\ref{Lprime8.5})$ with $(\ref{Lprime8.7})$, we get
\begin{align}\label{Lprime9}
\frac{\partial P}{\partial t}+N_{p,\Theta,\Ga}\int\limits_\Ga(u^p+v^p)d\sigma\le  C( p, h(T), L_1, D, \tilde D,  M, L_\Theta) +\tilde{N}_{p,\Theta,\Ga}\left(\int\limits_{\Omega}(u+v)\ dx \right)^p+N_{p,\Theta,\Ga}.
\end{align}
Finally, if we integrate over time, 
we find that $\|u\|_{p,\Ga_T}$ is bounded in terms of $p$, $\Ga$, $\Omega$, $\Theta$, $h(T)$, $w_1$, $w_2$ and $\|v\|_{p,\Omega_T}$. Since this holds for every natural number $p\ge 2$, we can use the assumption $(V_{Poly})$ and the bounds above, along with Proposition \ref{Holder} to conclude that $\|(u,v)\|_{\infty,\Omega_T}<\infty$. From Theorem \ref{local}, this contradicts our assumption that $T_{\max}<\infty$. Therefore, $T_{\max}=\infty$, and Theorem \ref{martinthm} is proved.
\hfill
\qed

For $m \geq 2$ components, we first obtain the following $L_p$ estimates. 
\begin{lemma} \label{lpestimate}
Suppose that $(V_N)$, $(V_F)$, $(V_{QP})$ and $(V_{L})$ are satisfied, and $u$ is the unique, componentwise nonnegative, maximal solution to $(\ref{grow})$.  If $1<p<\infty$ and $T=T_{max}<\infty$, then $\|u\|_{p,\Omega_T}$ and $\|u\|_{p,\Gamma_T}$ are bounded.
\end{lemma}
The proof of Lemma \ref{lpestimate} is using the Lyapunov function which is an extension of the polynomial (\ref{2poly}) to $m$ components, i.e., 
 \be\label{mpoly} P = \sum\limits_{|\beta|=  0}^p \frac{p!}{\beta ! ( p- |\beta|) !}\theta_1^{\beta_1^2}\ldots\theta_{m-1}^{\beta_{m-1}^2}u_1^{\beta_1}\ldots u_{m-1}^{\beta_{m-1}} u_m^{p - |\beta|} \ee where $|\beta| = |{\beta_1}|+ \ldots + |\beta_{m-1}|$ and $\beta! = \beta_1! \ldots \beta_{m-1}!$
The estimates are obtained following steps of Lemma 5.3 of \cite{V1} using this polynomial, which is relatively simpler than $H$ defined in \cite{RefWorks:8} and used  in  \cite{V1}. \\
{\bf{Proof of Theorem \ref{global}}:} From Theorem \ref{local}, we already have a componentwise nonnegative, unique, maximal solution of (\ref{grow}). If $T_{\max }= \infty$, then we are done. So, by way of contradiction assume $T_{\max }< \infty$.
  From Lemma $\ref{lpestimate}$ , we have $L_p$ estimates for our solution for all $p\geq 1$ on
$\Omega \times (0, T_{\max})$ and $M \times (0, T_{\max})$. We know from $(V_{Poly})$ that the $F_i$ and $G_i$ are polynomially
bounded above for each $i$. Then proceeding as in the proof of Theorem 3.3 in \cite{RefWorks:1} with the
bounds from Lemma $\ref{lpestimate}$ we have $T_{\max} = \infty$.

\section{Examples}\label{eg}
\subsection* {Example 1} Here, we give an example related to the well known Brusselator. Consider the system
\begin{align}\label{bruss}
u_{1_t}&=d_1 \Delta u_1 \quad & y\in \Omega_t, t>0  \nonumber\\ 
u_{2_t}&=d_2\Delta u_2 \quad & y\in \Omega_t, t>0 \nonumber\\
d_1\frac{\partial u_1}{\partial\eta_t}&=\alpha u_2-u_2^2u_1 \quad & y\in \partial\Omega_t, t>0\\\nonumber
d_2 \frac{\partial u_2}{\partial\eta_t}&=\beta-(\alpha+1)u_2+u_2^2u_1 \quad & y\in \partial\Omega_t, t>0\\\nonumber
u_i(y,0)&=w_i(y) & y\in \overline\Omega_0 \times\lbrace 0\rbrace\nonumber
\end{align}
where $d_1, d_2, \alpha, \beta>0$ and $w$ is sufficiently smooth and componentwise nonnegative. If we define \[ f(u)=\begin{pmatrix} 0\\0\end{pmatrix} \quad \text{and } \quad g(u)=\begin{pmatrix} \alpha u_2-u_2^2u_1\\\beta-(\alpha+1)u_2+u_2^2u_1\end{pmatrix} \]for all $u\in\mathbb{R}_+^2$, then $(V_N)$, $(V_{F})$, $(V_{QP})$ and $(V_{Poly})$ are satisfied with $a_1\geq 1$ and $L_a=\max\lbrace \beta,\alpha\cdot a_1\rbrace$. Therefore, Theorem \ref{global} implies (\ref{bruss}) has a unique, componentwise nonnegative, global solution.

\subsection*{Example 2} We next consider a general reaction mechanism of the form
\[ R_1+ R_2\substack{\longrightarrow\\ \longleftarrow} P_1\] 
where $R_i$ and $P_i$ represent reactant and product species, respectively. If we set $u_i=[R_i]$ for $i=1,2$, and $u_3 = [P_1] $, and let $k_f ,k_r$ be the (nonnegative) forward and reverse reaction rates, respectively, then we can model the process by the application of the law of conservation of mass and the second law of Fick (ﬂow) with the following reaction–diffusion system: 

\begin{align}\label{squares1}
u_{i_t}&=d_i \Delta u_i \quad & y\in \Omega_t, t>0, i=1,2,3 \nonumber\\ \nonumber
d_1\frac{\partial u_1}{\partial\eta_t}&=-k_fu_1u_2+k_ru_3\quad & y\in \partial\Omega_t, t>0\\
d_2 \frac{\partial u_2}{\partial\eta_t}&=-k_fu_1u_2+k_ru_3\quad & y\in \partial\Omega_t, t>0\\\nonumber
d_3 \frac{\partial u_3}{\partial\eta_t}&=k_f u_1u_2-k_ru_3 \quad & y\in \partial\Omega_t, t>0\\\nonumber
u_i(y,0)&=w_i(y) & y\in \overline\Omega_0 \times\lbrace 0\rbrace, i=1,2,3, \nonumber
\end{align}
where $d_i>0$ and the initial data $w$ is sufficiently smooth and componentwise nonnegative. If we define
\[ f(u)=\begin{pmatrix} 0\\0\\0\end{pmatrix} \quad \text{,} \quad g(u)=\begin{pmatrix}-k_fu_1u_2+k_rv_3\\ -k_fu_1u_2+k_rv_3\\k_f u_1u_2-k_rv_3   \end{pmatrix} \]
for all $u\in\mathbb{R}_+^3$, then $(V_N)$, $(V_{F})$, $(V_{QP})$ and $(V_{Poly})$ are satisfied. In addition, $(V_{L1})$ is satisfied with $L_1=0$ since \[ \frac{1}{2}f_1(z)+\frac{1}{2}f_2(z)+f_3(z)=0 \quad \text {and}\quad \frac{1}{2}g_1(z)+\frac{1}{2}g_2(z)+g_3(z)=0\] for all $z\in \mathbb{R}^3_{+}$. Therefore, the hypothesis of Theorems \ref{global} is satisfied. As a result (\ref{squares1}) has a unique, componentwise nonnegative, global solution. 

\subsection*{ Example 3} Finally, we consider a system that satisfies the hypothesis of the Theorem $\ref{global}$, where the boundary reaction vector field does not satisfy a linear intermediate sums condition. Let 

\begin{align}\label{square}
u_{1_t}&=d_1 \Delta u \quad & y\in \Omega_t, t>0  \nonumber\\ \nonumber
u_{2_t}&=d_2\Delta u \quad & y\in \Omega_t, t>0\\
d_1\frac{\partial u_1}{\partial\eta_t}&=\alpha u_1u_2^3-u_1u_2^2 \quad & y\in \partial\Omega_t, t>0\\\nonumber
d_2 \frac{\partial u_2}{\partial\eta_t}&=u_1u_2^2-\beta u_1u_2^6 \quad & y\in \partial\Omega_t, t>0\\\nonumber
u(y,0)&=w(y) & y\in \overline\Omega_0\times\lbrace 0\rbrace \nonumber
\end{align}
where $d_1, d_2, \alpha,\beta>0$ and $w$ is sufficiently smooth and componentwise nonnegative. In this setting 
\[ f(u)=\begin{pmatrix} 0\\0\end{pmatrix} \quad \text{,} \quad g(u)=\begin{pmatrix} \alpha u_1u_2^3-u_1u_2^2\\u_1u_2^2-\beta u_1u_2^6\end{pmatrix} \]
for all $u\in\mathbb{R}_+^2$. It is simple matter to see that $(V_N)$, $(V_{F})$, $(V_{QP})$ and $(V_{Poly})$ are satisfied. Also, if $a\ge 1$ then \[ a f_1(u)+f_2(u)=0 \quad \text {and}\quad ag_1(u)+g_2(u)\le (a\alpha-\beta)u_1(u_2^3-u_2^6)\le \frac{a\alpha}{4}u_1\] for all $u\in \mathbb{R}^2_{+}$. Consequenty, $(V_L)$ is satisfied. Therefore, Theorem \ref{global} implies $(\ref{square})$ has a unique, componentwise nonnegative, global solution.

\end{document}